\documentclass{amsart}
\usepackage{a4wide}

\usepackage[T1]{fontenc}
\usepackage[utf8]{inputenc}
\usepackage{amsmath,amssymb,amsthm}

\usepackage{array}
\usepackage{tabularx}
\usepackage{hyperref}
\usepackage{stmaryrd}
\usepackage{xcolor}
\renewcommand{\color}[1]{}
\usepackage{textcase}

\newtheorem{thm}{Theorem}[section]
\newtheorem{lem}[thm]{Lemma}
\newtheorem{prop}[thm]{Proposition}
\newtheorem{cor}[thm]{Corollary}

\theoremstyle{definition}
\newtheorem{defn}[thm]{Definition}

\theoremstyle{remark}
\newtheorem{rem}[thm]{Remark}

\title[Vanishing--Noise Asymptotics of Donsker--Varadhan Rate Functions]{Vanishing--Noise Asymptotics of Donsker--Varadhan Rate Functions on the Circle}

\author{Milan Koresski}
\address{Department of Mathematics,
CY Cergy Paris Université,
2 avenue Adolphe Chauvin,
95302 Cergy-Pontoise, France}
\email{milankoresski@wanadoo.fr}


\begin{document}

\begin{abstract}
We study the vanishing noise limit of the rate function for the Donsker-Varadhan large deviation principle for one-dimensional diffusion processes on a circle. As is well known, the rate function can be represented either as the Legendre transform of {\color{red}principal} eigenvalue of the perturbed infinitesimal generator or by a variational formula. We analyze the asymptotic behavior of the {\color{red}principal} eigenvalue as the parameter in front of the noise goes to zero and compare the Legendre transform of the limit with the expression obtained as the limit of the variational representation. We prove, in particular, that the resulting expressions do not always coincide, leading to continuity and discontinuity phenomena in infinite-dimensional functional spaces. {\color{red}Moreover, we use the previous analysis to pass to the limit in the LDP, using the notion of $\Gamma$-convergence.}
\end{abstract}

\maketitle
{\color{red}
\tableofcontents}
\newpage
\section{Introduction}

Let \((X_t^\varepsilon)_{t \ge 0}\) be the family of one-dimensional diffusion processes on the unit circle \(\Pi\) defined as the solution of the stochastic differential equation
\begin{equation*}
    dX_t^\epsilon = \sqrt{\epsilon}\,\sigma(X_t^\epsilon)\,dW_t + b(X_t^\epsilon)\,dt,
\end{equation*}
where $b$ and $\sigma$ are $C^2(\Pi)$ functions on $\Pi$, $W_t$ denotes the standard Brownian motion, and $\epsilon > 0$ is a small noise amplitude. Our goal is to study the behaviour of the associated large-deviation rate functions in the small-noise limit $\epsilon \to 0$.

\subsection{Empirical measure and large deviations}

For $T > 0$, we define the empirical measure of the diffusion by
\begin{equation*}
    \mu_T^\epsilon = \frac{1}{T} \int_0^T \delta_{X_t^\epsilon}\,dt.
\end{equation*}
{\color{red}
It is a classical result in Donsker--Varadhan \cite{DonskerVaradhan1975} (Theorem 3, (1.11) and (1.12), pages 5-6) that the family $(\mu_T^\epsilon)_{T>0}$ satisfies a Large Deviation Principle (LDP) on the space of probability measures on $\Pi$, with speed $T$ and rate function
\begin{equation*}
    I_\epsilon(\mu)
    = \sup_{\substack{u > 0 \\ u \, \in \, C^2(\Pi)}}
        - \int_{\Pi} \frac{\mathcal{L}_\epsilon u}{u} \, d\mu. 
\end{equation*}
Here $\mathcal{L}_\epsilon$ denotes the infinitesimal generator of $X_t^\epsilon$, defined on $C^2(\Pi)$ by

\begin{equation*}
    \mathcal{L}_\epsilon u(x) = \epsilon\, a(x)\,u''(x) + b(x)\,u'(x) \,\,\, \text{with} \,\,\, a(x) := \sigma(x)^2/2.
\end{equation*}}

It has been established, see \cite{FengKurtz2006} (Proposition 3.8, page 45), that this variational expression is equivalent to a Legendre-Fenchel transform representation:
\begin{equation*}
    I_\epsilon(\mu)
    = \sup_{c \,\in \, C^0(\Pi)}
    \left\{
        \int_{\Pi} c \, d\mu
        - \Lambda_\epsilon(c)
    \right\},
\end{equation*}
where
\begin{equation*}
    \Lambda_\epsilon(c)
    := \lim_{T \to \infty}
       \frac{1}{T} \log \mathbb{E}_x \left[
           \exp\left( \int_0^T c(X_s^\epsilon) \, ds \right)
       \right].
\end{equation*}
Moreover, it is well known, see \cite{DonskerVaradhan1976} ((1.2), page 1), that $\Lambda_\epsilon(c)$ coincides with $\lambda^c_\epsilon$, the {\color{red}principal} eigenvalue {\color{red}(the eigenvalue with the maximal real part)} of the perturbed operator
\begin{equation*}
    L_\epsilon := \mathcal{L}_\epsilon + c.
\end{equation*}

\subsection{The vanishing-noise problem}

Our aim is to characterize the asymptotic behaviour of the rate function in the regime where the noise parameter $\varepsilon$ tends to $0$. Formally, for any probability measure $\mu \in P(\Pi)$ we define
\begin{equation*}
    I(\mu) := \lim_{\epsilon \to 0} I_\epsilon(\mu) = \lim_{\epsilon \to 0} \sup_{\substack{u > 0 \\ u \, \in \, C^2(\Pi)}}
        - \int_{\Pi} \frac{\mathcal{L}_\epsilon u}{u} \, d\mu
    ,\,\,\, \text{and} \,\,\, \Lambda(c) := \lim_{\epsilon \to 0} \Lambda_\epsilon(c),
\end{equation*}
whenever these limits exist, and
\begin{equation*}
    J(\mu)
    := \sup_{c \, \in \, C^0(\Pi)}
       \left\{
           \int_{\Pi} c \, d\mu
           - \Lambda(c)
       \right\}.
\end{equation*}
Taking $0$ as a test function in the above variational expressions, one can prove that both $I_\epsilon$ and $I$ are nonnegative, and likewise $J(\mu) \ge 0$ for all probability measures $\mu$.

\vspace{\baselineskip}

The existence of the limit $I(\mu)$ for an arbitrary measure $\mu$ is not immediate, nor is the relation between the functions $I$ and $J$: the functional $I$ is the limit of Legendre transforms, whereas $J$ is the Legendre transform of the limit. Thus, the central question is whether the limit defining $I$ exists and whether, in this setting, the Legendre transform commutes with the small-noise limit; in other words, whether
\begin{equation}\label{1}
    \lim_{\epsilon \to 0} I_\epsilon(\mu)
    \stackrel{?}{=}
    \sup_{c \, \in \, C^0(\Pi)}
    \left\{
        \int_{\Pi} c \, d\mu
        - \lim_{\epsilon \to 0} \Lambda_\epsilon(c)
    \right\}.
\end{equation}

{\color{red}
Moreover, after addressing this question in different settings, we show that, in some of these cases, the previously obtained results, combined with the notion of $\Gamma$-convergence, allow us to establish large deviation results in the double limit for the empirical measure
$\mu_T^\epsilon$.}

\subsection{Strategy and main results}

In a first step, to compare the two functionals $I$ and $J$, we carry out a detailed asymptotic analysis of $\lambda_\epsilon^{\,c}$, the principal eigenvalue of the operator
\begin{equation}
\label{2}
    L_\epsilon(u_\epsilon)
    := \epsilon\, a(x)\, u_\epsilon'' + b(x)\, u_\epsilon' + c(x)\, u_\epsilon
    = \lambda_\epsilon^{\,c}\, u_\epsilon,
\end{equation}
where we have made explicit the dependence of the principal eigenvalue on the potential~$c$. For the sake of clarity, this dependence will be omitted in the notation from now on. This study is based on the same arguments as in \cite{PiatnitskiRybalko2015} and will allow us to obtain an explicit formula for $J$. In a second step, we identify the cases in which $I$ can be computed explicitly, and employ these expressions for comparing it with $J$. {\color{red} Finally, we show that $I_\varepsilon$ $\Gamma$-converges to $J$, and establish large deviation principles in the double limit.}

\vspace{\baselineskip}

As we shall see, the nature of the asymptotic regime as $\epsilon \to 0$ depends crucially on the qualitative properties of the drift $b$. Different behaviours arise depending on whether $b$ is separated from $0$, whether it vanishes, or whether it changes sign. These structural features govern the geometry of the deterministic flow and, in turn, the spectral asymptotics of the principal eigenvalue in \eqref{2}, leading to different explicit expressions for the functions $I$ and $J$. We shall therefore distinguish three regimes in our analysis:
\begin{itemize}
    \item[(H1)] the fully degenerate case where $b \equiv 0$;
    \item[(H2)] the nondegenerate case where $b$ is separated from $0$ on $\Pi$;
    \item[(H3)] the sign-changing case, in which $b$ takes both positive and negative values and admits a finite number of zeros.
\end{itemize}

Many results have already been established concerning the asymptotic analysis of the {\color{red}principal} eigenvalue; under assumptions (H1) and (H2), the asymptotic behaviour is now regarded as classical. Here we present some known results, together with their sources, in a setting of assumptions comparable to (H3).

\vspace{\baselineskip}

The paper by Peng, Zhang and Zhou \cite{PengZhangZhou2020} provides a detailed analysis of the asymptotic behaviour of the {\color{red}principal} eigenvalue of a second–order elliptic operator in the regime of small and large diffusion when the field $b$ is the gradient of a function $m$. Indeed, Theorem 1.1 and (1.2), which deal with the case of N-dimensional operators of type (1.1), provide similar results to the ones we obtain. However, the hypotheses require that $\det(D^{2} m(x)) \neq 0$. In our setting, their conditions correspond precisely to the case of a finite number of hyperbolic zeros. Similarly, the article \cite{PiatnitskiRybalko2015}  by A. Piatnitski, A. Rybalko and V. Rybalko proves a comparable result under the same hyperbolicity assumption on the zeros and establish a similar formula (see (2.15) and (2.18)). In \cite{Kifer1980} and \cite{Kifer1990}, Kifer obtained several fundamental results on the asymptotic behaviour of the principal eigenvalue in the small--diffusion regime however, the analysis strongly relies on hyperbolicity assumptions on the critical points or invariant sets of the drift (see Condition (C) and (2.3) in Theorem 2.1 in  \cite{Kifer1980}). In our paper, in the one dimensional setting, a suitable generalization of the techniques introduced in \cite{PiatnitskiRybalko2015} makes it possible to obtain the asymptotic description even when the drift has zeros of arbitrary order and flat zeros. This yields a genuine extension of the previously known results in the case of degenerate critical points. 

\vspace{\baselineskip}
{\color{red}
Our work follows an approach analogous to that of \cite{FaggionatoGabrielli2010}. In this spirit, we establish a large deviation principle (LDP) for the empirical measure in the vanishing noise limit regime within a degenerate framework, made possible by the one-dimensional setting in which we work. In  \cite{FaggionatoGabrielli2010}, they study the large deviation principle for the invariant measure in the vanishing noise regime at speed $1/\varepsilon$. Their approach relies on viscosity solution techniques for Hamilton--Jacobi equations, in a way that is closely related to the methods used in our first part to determine the principal eigenvalue. This allows them to identify the rate function as a solution of a stationary Hamilton--Jacobi equation (2.20) given explicitly by (2.12) in Theorems 2.3 and 2.4.}

\vspace{\baselineskip}

{\color{red}
In \cite{DiGesuMariani2017}, more refined convergence results are established in the hyperbolic case for the asymptotic behavior of $I_\varepsilon$ via $\Gamma$-convergence, culminating in an asymptotic expansion of the rate function setting and yielding several large deviation principle estimates (see (1.5) or (2.19) and Theorem 2.3). In a manner analogous to this article, we employ the notion of $\Gamma$-convergence to obtain a generalization to the one-dimensional case when the field admits degenerate zeros of the results stated in Theorem 2.3 (see  (2.14)) and Corollary 2.5 (see (2.24)).  More precisely, in their case, coercivity conditions are imposed on the drift in order to control the trajectories; see (A.2) and (A.3) and Condition A.1 requires the drift $b$ to be the gradient of a $C^\infty$ Morse function and set the work in a hyperbolic framework.

\vspace{\baselineskip}

In \cite{BertiniGabrielliLandim2023}, the authors establish a large deviation principle for some empirical measure associated with a diffusion process on $\mathbb{R}^n$ in the joint regime where the time diverges and the noise intensity vanishes. More precisely, in their study, they show that the family of probability measures describing the empirical distribution of trajectories satisfies a large deviation principle with speed $\frac{T}{\varepsilon}$ (see Theorem 2.2), uniformly with respect to the initial condition in compact sets. The associated rate function is given by the expectation, under the considered measure, of the Freidlin--Wentzell functional per unit time. We obtain a large deviation principle with speed $T$ for the classical empirical measure, which only captures occupation times. By contrast, their empirical measure averages over time-shifted trajectories, thereby leading to a different framework. Finally, in \cite{Raquepas2024}, the entropy production associated with simimar diffusion processes is studied and large deviation principle estimates for this entropy production are obtained in Theorem 6.5 in the small-noise asymptotics and in hyperbolic settings (see Assumption (ND)).}

\begin{table}[h!]
\centering
\tiny
\renewcommand{\arraystretch}{1.5}

\begin{tabular}{|c|
    >{\centering\arraybackslash}m{3.2cm}|
    >{\centering\arraybackslash}m{7.5cm}|}
\hline
\textbf{Hyp.} 
  & \textbf{Limit of $\lambda_\epsilon$} 
  & \textbf{Behaviour of $I(\mu)$ and $J(\mu)$}
  \tabularnewline
\hline

\textbf{(H1)}
& $\displaystyle \max_{x \in \Pi} c(x)$
&
$\displaystyle
\begin{gathered}
J(\mu)=0 \quad \text{for all }\mu\in P(\Pi),\\[0.4em]
I(\mu)=
\begin{cases}
0 & \mu = f\,dx,\; f>0,\ f\in C^1(\Pi),\\[0.2em]
+\infty & \text{in some cases of absolutely continuous measures } \mu,\\[0.2em]
+\infty & \mu \text{ has an isolated atom},\\[0.2em]
+\infty & \text{in some cases of singular continuous measures } \mu.
\end{cases}
\end{gathered}$
\tabularnewline
\hline

\textbf{(H2)}
&
$\displaystyle
\left(\int_{\Pi} \frac{c(s)}{b(s)}\,ds\right)
\Big/
\left(\int_{\Pi} \frac{1}{b(s)}\,ds\right)$
&
$\displaystyle
\begin{gathered}
J(\mu)=
\begin{cases}
0 & \mu=\omega\,ds,\quad
\omega := \bigl(\tfrac{1}{b}/\int_\Pi \tfrac{1}{b}\,dx\bigr),\\
+\infty & \text{otherwise},
\end{cases}\\[0.6em]
I(\mu)=
\begin{cases}
0 & \mu=\omega\,ds,\\
+\infty & \text{otherwise}.
\end{cases}
\end{gathered}$
\tabularnewline
\hline

\textbf{(H3)}
&
$\displaystyle
\max_{\substack{x,\, b(x)=0}}
\bigl(-(0\lor b'(x))+c(x)\bigr)$
&
$\displaystyle
\begin{gathered}
J(\mu)=
\begin{cases}
\sum_{i=1}^n a_i\,(0\lor b'(x_i))
& \operatorname{supp}\mu\subset \{b=0\},\\[0.3em]
+\infty & \text{otherwise},
\end{cases}\\[0.6em]
I(\mu)=+\infty.
\end{gathered}$
\tabularnewline
\hline
\end{tabular}
\end{table}

Each of these regimes exhibits a distinct small-noise asymptotic behaviour for the operator $L_\epsilon$ and therefore requires a separate analysis for the comparison of $I$ and $J$. In the case (H1), the Legendre transform does not commute with the limit $\epsilon \to 0$, since $I$ and $J$ differ on a wide range of probability measures. Indeed we do not have equality on Dirac masses, on some absolutely continuous measures, and also on some singular continuous ones. However, in the case (H2), one exhibits a totally different behavior as $I = J$. Finally under the hypothesis (H3), the limit of $\lambda_\epsilon$ is determined by the behaviour of $b$ near its zeros, and $J$ is finite only for atomic measures supported on these points, with a cost depending on their local stability. By contrast, $I$ always diverges, thus the functions differ on finite values other than only $0$ and $+\infty$.

\vspace{\baselineskip}

Finally, we examine the mixed case in which $b$ changes sign and may admit a finite number of intervals of vanishing and zeros. These results can be viewed mainly as a generalisation and a synthesis of the results obtained in the previous table under assumptions (H1) and (H3). In this setting as well, the functionals $I$ and $J$ differ on a large class of probability measures supported on the set $\{x, b(x)=0\}$. This includes, for instance, measures with isolated atoms, certain absolutely continuous measures, and certain measures with a singular continuous part. In addition, we present partial results concerning the
asymptotics of the associated eigenfunction.

\section{The case where {\color{red}{\normalfont \textit{b}}}  is null}\label{sec:2}
\subsection{ Asymptotic of the {\color{red}principal} eigenvalue}

In this section, we are interested in the asymptotic analysis of the {\color{red}principal} eigenpair for the singularly perturbed
spectral problem \eqref{2} with $a\in C^2(\Pi)$, $b$ equal to 0 and $c\in C^0(\Pi)$. Our operator becomes
\begin{equation}\label{3}
(\textit{L}_\epsilon)(u_{\epsilon}) :=\epsilon \, a(x) u_\epsilon'' + c(x)u_{\epsilon} = \lambda_\epsilon u_{\epsilon}.
\end{equation}

We state the main result of this section in the form of the following proposition.

\begin{prop}\label{prop:2.1}
  Under the assumptions stated above,
\begin{equation}\label{4}
\lim\limits_{\epsilon \rightarrow 0} \lambda_\epsilon = \max\limits_{x \,\in \,\Pi } c(x).
\end{equation} 
\end{prop}

\begin{proof}
We begin the proof by showing that $(\lambda_\epsilon)_\epsilon$ is bounded. Indeed, using the maximum and minimum points of the twice differentiable function $u_\epsilon$, we establish using $\eqref{3}$ that for any $\epsilon > 0$, 
    \begin{equation*}
 \min\limits_{x \,\in\, \Pi} c(x) \leq \lambda_\epsilon \leq \max\limits_{x\, \in \,\Pi} c(x).
\end{equation*}

{\color{red} If $c$ is constant, there is nothing to prove. Hence, we assume that it is not constant.} To prove the lower bound, it suffices to remark that $\textit{L}_\epsilon$ is a self-adjoint operator in $L^2(\Pi)$ endowed with the inner product
\[
\langle u, v \rangle := \int_\Pi \frac{u(x)\, v(x)}{a(x)} \, dx.
\]

The characterization of the {\color{red}principal} eigenvalue given by the Rayleigh–Ritz formula \cite{Evans2010} (Theorem 2 (Variational principle for the principal eigenvalue), page 336, applied to operator (3), page 294) yields
\begin{equation*} 
\lambda_\varepsilon
= \sup_{\substack{v \in H^1(\Pi), \, \|v\|_{L^2(\Pi)}=1}}
\int_0^1 \Big(-\varepsilon \, a(x) \,|v'(x)|^2 + c(x)\,|v(x)|^2\Big)\,dx.
\end{equation*}

As $c$ is continuous, for every \(\delta > 0\),
the set $E_\delta := \{ x \in \Pi, \, c(x) > \max_{x \in \Pi} c(x) - \delta \}$ is a countable union of disjoint open intervals. Thus we can choose, \(v_\delta \in H^1_{\Pi}\) such that $\operatorname{supp}(v_\delta) \subset E_\delta$ and $\|v_\delta\|_{L^2} = 1$. Taking that function in the supremum, we obtain
\begin{equation*}
     \max\limits_{x\, \in \,\Pi} c(x) - \delta \leq \liminf_{\varepsilon \to 0} \lambda_\varepsilon.
\end{equation*}

Passing in the limit as $\delta$ tends to $0$, we obtain the desired result.
\end{proof}

\subsection{Explicit calculation of J}

\begin{prop}\label{prop:2.2}
    In our condition for any probability measure $\mu \, \in \, P(\Pi)$ we have
    \begin{equation*}
        J(\mu) = 0.
    \end{equation*}
\end{prop}

\begin{proof}

Let us fix a measure $\mu$. By Proposition \ref{prop:2.1},
\begin{equation*}
J(\mu) = \sup\limits_{c \,\in \, C^0(\Pi) } \left\{ \int_\Pi c \, d\mu - \Lambda(c)\right\} = \sup_{c \,\in \, C^0(\Pi)} \left\{ \int_{\Pi} c \,d\mu - \max\limits_{x \in \Pi } c(x)\right\}.
\end{equation*} 

Moreover, for any \(c \in C^{0}(\Pi)\)
\begin{equation*}
\int_\Pi c \, d\mu - \Lambda(c) = \int_{\Pi} c\, d\mu - \max\limits_{x \in \Pi } c(x) \leq 0.
\end{equation*}

Choosing \(c\equiv 0\) yields
\[
\int_{\Pi} c\, d\mu - \Lambda(c)=0,
\]
therefore \(I(\mu)=0\) as the supremum is attained at \(c\equiv 0\).
\end{proof}

\subsection{Calculation of I}

We are interested in the asymptotic behavior of 

\begin{equation*}
I_\varepsilon(\mu) := \sup\limits_{u > 0, \, u \, \in \, C^2(\Pi) } - \int_{\Pi} \varepsilon  \, a(x) \frac{u''(x)}{u(x)} \, d\mu(x)  
\end{equation*}
as $\varepsilon$ goes to $0$. Using the change of variable $u= exp(w)$ where $w$ is a $C^2(\Pi)$ function we get 
\begin{align}
\label{5}
I_\epsilon(\mu) &= \sup\limits_{w \in C^2(\Pi) } - \int_{\Pi} \epsilon \; a(x)( w''(x)+ w'(x)^2) \, d\mu(x) = \sup\limits_{v \in \Tilde{C}^{1,0}(\Pi)}\phi_{\mu,\epsilon}(v) ,
\end{align}
where $\Tilde{C}^{1,0}(\Pi) = \lbrace v \in C^1(\Pi), \int_{\Pi} v(x) dx= 0  \rbrace$ and $\phi_{\mu,\epsilon}(v) = - \int_{\Pi} \epsilon \, a\left( v'+ v^2\right) \, d\mu.$

\vspace{\baselineskip}

Remarking that $I(\mu)=\lim\limits_{\epsilon \rightarrow 0}\epsilon\, I_{1}(\mu),$ the limit \(I(\mu)\) can only take the values \(0\) or \(+\infty\) and the following characterization of this limit holds
\begin{equation*}
I(\mu) = \begin{cases}
\;\;\; 0 & \displaystyle if \;\;\; I_1(\mu) < +\infty,\\
+\infty & if \;\;\; I_1(\mu) = +\infty.
\end{cases}
\end{equation*} 

Thus, to analyse the limit $I$, it suffices to study the behaviour of \(I_1\).

\begin{prop}\label{prop:2.3}
    Let $\mu \in P(\Pi)$ be a probability measure possessing an isolated atom. Then
        \begin{equation*}
I_1(\mu) = + \infty.
\end{equation*}
\end{prop}

\begin{proof}
    We first suppose that $\mu$ possess an atom at $x_0=0$ and we set $\delta > 0$ small enough such that $\mu_0(B(0,\delta))= \mu_0(0)$ and $\mu_0(B(0,2\delta)\setminus B(0,\delta))= 0$. We define for any $\epsilon > 0$ the periodic function $v_\epsilon$ by
        \[ v_\epsilon(x) = \begin{cases} 
          - sin(x/ \epsilon) & x\in B(0,\delta) \\
          r_\epsilon(x) & x \in B(0,2\delta) \setminus B(0,\delta) \\
          0 & x \in \Pi \setminus B(0,2\delta),
       \end{cases}
    \]
    where $r_\epsilon$ is a well-chosen function that makes the function in $\Tilde{C}^{1,0}(\Pi)$. Then
\begin{align*}
     I_1(\delta_{0}) \ge - \, a(0)\bigl( v_\epsilon^{\prime}(0) + v_\epsilon(0)^2 \bigl)=  a(0)/\epsilon \underset{\epsilon\to 0}{\longrightarrow}\; +\infty
\end{align*}
which ends the proof of the proposition.
\end{proof}

\begin{prop}\label{prop:2.4}
Let  $\mu \in P(\Pi)$ be a probability measure which is absolutely continuous with respect to the Lebesgue measure, with a positive density \(f \in C^{1}(\Pi)\). Then
\[
I_1(\mu) = \frac{1}{4} \int_{\Pi} \frac{\bigl( (a f)' \bigr)^{2}}{a f}\, dx.
\]
\end{prop}

\begin{proof}
We set 
\begin{equation*}
     v_0 = \frac{(af)'}{2af} = \frac{1}{2}\ln(af)' \in \Tilde{C}^{1,0}(\Pi).
\end{equation*}
Then, computing we get

\begin{equation*}
    \phi_{\mu,1}(v_0) = \int_\Pi a \, v_0^2 d\mu = \frac{1}{4} \int_{\Pi} \frac{\bigl((a f)'\bigr)^{2}}{a f}\, dx.
\end{equation*}

Moreover, one can show that the supremum is attained at this point, as for any $v \in \Tilde{C}^{1,0}(\Pi)$, the Young Inequality:
\begin{equation*}
    -\int_\Pi av' \, f\, dx= \int_\Pi (af)'v \, dx= \int_\Pi \frac{(af)'}{\sqrt{2af}}\, \times \sqrt{2af} \;v \, dx \leq \frac{1}{2} \int_\Pi \frac{(({af})')^2}{2af}\, dx + \frac{1}{2} \int_\Pi 2 \,a v^2 \, f\, dx
\end{equation*}
gives us
\begin{equation*}
    -\int_\Pi av' +av^2 \, d\mu \leq \frac{1}{4} \int_\Pi \frac{{((af)')}^2}{af} \, dx,
\end{equation*}
which ends the proof.
\end{proof}

\begin{rem}\label{rem:2.5}
    This result is a particular case of the one obtained in \cite{DonskerVaradhan1975} ((1.4), page 3).
\end{rem}

The next proposition gives a partial result for some continuous singular measures and more general ones, exploiting the notion of local dimension of a measure. 

\begin{defn}\label{defn:2.6}
We define the left and right lower and upper local dimensions {\color{red} of a measure $\mu \in P(\Pi)$ at point $x_0 \in \Pi$} by

\[
\dim_{\mathrm{loc}}^{-,l}(\mu, x_0)
    := \liminf_{r \to 0} \frac{\log \mu\bigl([x_0 - r,\,x_0]\bigr)}{\log r},
    \,\,\,
\dim_{\mathrm{loc}}^{+,l}(\mu, x_0)
    := \limsup_{r \to 0} \frac{\log \mu\bigl([x_0 - r,\,x_0]\bigr)}{\log r},
\]
and
\[
\dim_{\mathrm{loc}}^{-,r}(\mu, x_0)
    := \liminf_{r \to 0} \frac{\log \mu\bigl([x_0,\,x_0 + r]\bigr)}{\log r},
    \,
\dim_{\mathrm{loc}}^{+,r}(\mu, x_0)
    := \limsup_{r \to 0} \frac{\log \mu\bigl([x_0,\,x_0 + r]\bigr)}{\log r}.
\]

If the lower and the upper limits coincide, we say that $\mu$ has a \emph{left} or a \emph{right} local dimension at $x_0$, and we denote this common value by $\dim_{\mathrm{loc}}^{l}(\mu, x_0)$ and $\dim_{\mathrm{loc}}^{r}(\mu, x_0)$. Moreover, if these two limits exist and coincide, we say that 
$\mu$ has the local dimension at $x_0$, given by $\dim_{\mathrm{loc}}(\mu,x_0)=\dim_{\mathrm{loc}}^{l}(\mu, x_0) = \dim_{\mathrm{loc}}^{r}(\mu, x_0)$. Finally, we also adopt the convention that if in the numerator 
$\mu(\cdot)$ appearing in any of the above definitions is identically equal to $0$ for all sufficiently small $r$, then the corresponding local dimension is defined to be $-\infty$.

\end{defn}

\begin{lem}\label{lem:2.7}
Let $\mu \in P(\Pi)$ be a probability measure that has no atom at the point $x_0$. Then
    \begin{align*}
I_1(\mu) &= \sup\limits_{v \in \Tilde{C}^{1,0}_{p, x_0}(\Pi)}\phi_{\mu,1}(v) ,
\end{align*}
where $\Tilde{C}^{1,0}_{p, x_0}(\Pi)$ denotes the set of functions that are in $C^0(\Pi)\cap C^1(\Pi \setminus \lbrace x_0\rbrace)$ and having zero mean over $\Pi$.
\end{lem}

\begin{proof}
    First of all it is clear that $\Tilde{C}^{1,0}(\Pi) \subset \Tilde{C}^{1,0}_{p, x_0}(\Pi)$, thus 
        \begin{align*}
I_1(\mu) \leq \sup\limits_{v \in \Tilde{C}^{1,0}_{p, x_0}(\Pi)}\phi_{\mu,1}(v) := \Tilde{I_1}(\mu).
\end{align*}
    
We now prove the reverse inequality. For any $\epsilon$, let us fix $v_\epsilon \in \Tilde{C}^{1,0}_{p, x_0}(\Pi)$ such that $ \Tilde{I_1}(\mu)  - \epsilon < \phi_{\mu,1}(v_\epsilon)$. Our goal is to show that for any $\epsilon$, there exists a sequence $v_n$ of function in $\Tilde{C}^{1,0}(\Pi)$ such that $\lim_{n \rightarrow +\infty} \phi_{\mu,1}(v_n) = \phi_{\mu,1}(v_\epsilon)$. 
    
\vspace{\baselineskip}

Without loss of generality, we suppose that $x_0 = 0$. We set for any fixed $\epsilon$

\begin{equation*}
    \Tilde{v}_\delta(x) = v_\delta(x) - \int_\Pi v_\delta(s) \, ds \,\,\, \text{where} \,\,\, v_\delta(x) := \int_0^x (v_\epsilon' * \omega_\delta)(s) \, ds,
\end{equation*}
and $(\omega_\delta)_\delta$ denotes a family of nonnegative $C^1$ functions with compact support contained in $B(0,\delta)$, normalized to have unit mass. By construction, it is clear that $\Tilde{v}_\delta \in \Tilde{C}^{1,0}(\Pi)$ and that 
\begin{equation*}
\sup_{\Pi} \Bigl(
\lvert \tilde{v}_\delta - v_\varepsilon \rvert
+ \lvert \tilde{v}_\delta' - v_\varepsilon' \rvert
\Bigr)
\leq
\sup_{\Pi} \Bigl(
\lvert v_\delta - v_\varepsilon \rvert
+ \lvert v_\varepsilon' * \omega_\delta - v_\varepsilon' \rvert
\Bigr)
+ \left| \int_{\Pi} v_\delta(s)\,ds \right|
\;\xrightarrow[\delta \to 0]{}\; 0,
\end{equation*}
from which the result follows.
\end{proof}

\begin{prop}\label{prop:2.8}
    Let $\mu \in P(\Pi)$ be a probability measure such that there exist two points \( x_0 \) and \( y_0 \) for which
\begin{equation*}
    \dim_{\text{loc}}^{-,l}(\mu, x_0) := \alpha^-_l, \,\dim_{\text{loc}}^{-,r}(\mu, y_0) := \beta^-_r, \,
    \dim_{\text{loc}}^{+,r}(\mu, x_0) := \alpha^+_r, \,\dim_{\text{loc}}^{+,l}(\mu, y_0) := \beta^+_l \,\,
\end{equation*}
with
    \[
\min(\alpha_l^{-},\,\beta_r^{-}) \;<\; \min(\beta_l^{+},\,\alpha_r^{+})
\,\,\,\text{and}\,\,\,
\min(\alpha_l^{-},\,\beta_r^{-}) \;<\; 1.
\]

    Then $I_1(\mu) = + \infty$.
\end{prop}

\begin{proof}
    First of all, without loss of generality, we suppose that $x_0 = 0$. Our goal is to construct a good sequence of test functions that makes the supremum diverge. To construct such test functions, let us pose for any negative number $A$
    \begin{align*}
    f(x) =
\begin{cases}
A(x+\frac{1}{2}) & \text{if } x \in [-\frac{1}{2}, 0], \\
A(\frac{1}{2} - x) & \text{if } x \in (0, \frac{1}{2}], \\
0 & \text{otherwise }
\end{cases}
\end{align*}
and $\epsilon > 0$ small enough
\begin{equation*}
v_\varepsilon(x) =
\begin{cases}
\displaystyle f\left( \frac{x}{\varepsilon} \right) & \text{if } x \in [-\varepsilon, \varepsilon], \\[10pt]
\displaystyle -f\left( \frac{x - y_0}{\varepsilon} \right) & \text{if } x \in [y_0-\varepsilon, y_0 + \varepsilon], \\[10pt]
0 & \text{otherwise}.
\end{cases}
\end{equation*}

It is clear that $v_\varepsilon$ is a periodic function in $\Tilde{C}^{1,0}_{p, x_0}(\Pi)$ as a piecewise continuous function and is of zero integral as the $f$ and $-f$ compensate each other on $[-\epsilon, \epsilon]$ and $[y_0- \epsilon, y_0 + \varepsilon]$. Using Lemma \ref{lem:2.7}, \(v_\varepsilon\) can be used as an admissible test function in the supremum, even though $v_\varepsilon\notin \Tilde{C}^{1,0}(\Pi)$. Thus for $\epsilon$ small enough we get

\begin{align*}
I_1(\mu) &\ge - \int_{-\epsilon}^\varepsilon a(x) \bigg( \frac{1}{\epsilon}
f'\left( \frac{x}{\varepsilon} \right)
+  f\left( \frac{x}{\varepsilon} \right)^2\bigg)
 d\mu(x) \\
 &- \int_{y_0- \epsilon}^{y_0+ \epsilon} a(x) \bigg(\frac{-1}{\epsilon}
f'\left( \frac{x-y_0}{\varepsilon} \right)
+  f\left( \frac{x-y_0}{\varepsilon} \right)^2\bigg)
 d\mu(x) \\
 &\sim -\int_{-\epsilon}^\varepsilon \frac{1}{\epsilon} \, a(x) \,
f'\left( \frac{x}{\varepsilon} \right) d\mu(x) + \int_{y_0- \epsilon}^{y_0+ \epsilon} \frac{1}{\epsilon} \, a(x) \,
f'\left( \frac{x- y_0}{\varepsilon} \right) d\mu(x)\\
& = \frac{A}{\varepsilon}\Bigg(
-\int_{-\epsilon/2}^{0} a(x)\, d\mu(x)
+\int_{y_0-\epsilon/2}^{y_0} a(x)\, d\mu(x)
+
\int_{0}^{\epsilon/2} a(x)\, d\mu(x)
\;-\;
\int_{y_0}^{y_0+\epsilon/2} a(x)\, d\mu(x)
\Bigg)\\
& \ge \frac{A}{\varepsilon}\Bigg(
\max a \Big(
\mu\bigl([y_0-\epsilon/2,y_0]\bigr)
+ \,\mu\bigl([0,\epsilon/2]\bigr)
\Big) - \min a \Big(
\mu\bigl([-\epsilon/2,0]\bigr)
+ \,\mu\bigl([y_0,y_0+\epsilon/2]\bigr)
\Big)
\Bigg)\\
& \sim \frac{A}{\varepsilon}\Bigg(
\max a \Big(
C_1^+(\epsilon/2)^{\beta^+_l}
+ C_2^+\,(\epsilon/2)^{\alpha^+_r}
\Big) - \min a \Big(
C_3^+(\epsilon/2)^{\alpha^-_l}
+ C_4^+\,(\epsilon/2)^{\beta^-_r}
\Big)
\Bigg) \\
& \sim - \, \frac{A}{\varepsilon} \min a \Big(
C_3^+ (\epsilon/2)^{\alpha^-_l}
+ C_4^+\,(\epsilon/2)^{\beta^-_r}
\Big).
\end{align*}

The result follows from the second assumption by taking the limit as $\ epsilon \ to 0$.  \end{proof}

\begin{cor}\label{cor:2.9}

     Let $\mu \in P(\Pi)$ be a probability measure such that there exist two points \( x_0 \) and \( y_0 \) for which
    \begin{equation*}
        \dim_{\text{loc}}(\mu, x_0) = \alpha < 1 \,\,\, \text{and} \,\,\, \alpha < \beta = \dim_{\text{loc}}(\mu, y_0).
    \end{equation*}
    
     Then $I_1(\mu) = + \infty$.
\end{cor}

\begin{proof}
   We consider the same test function $v_\varepsilon$ as previously, except that we now replace in its construction $f$ by
        \begin{align*}
    f_\delta(x) =
\begin{cases}
A(x+\delta) & \text{if } x \in [-\delta, 0], \\
A\delta - \dfrac{A\delta}{1 - \delta} \, x & \text{if } x \in (0, 1-\delta], \\
0 & \text{otherwise},
\end{cases}
\end{align*}
with $ 0< \delta < 1$. The same calculation gives us 
\begin{align*}
    &I_1(\mu) \gtrsim \frac{A}{\varepsilon}\Bigg(
\max a \Big( C_1^+
(\varepsilon\delta)^{\beta}
+ \frac{C_2^+ \, \delta}{1-\delta}\,(\varepsilon(1-\delta))^{\alpha}
\Big) - \min a \Big(
C_3^+(\epsilon \delta)^{\alpha}
+ \frac{C_4^+ \, \delta}{1-\delta}\,(\varepsilon(1-\delta))^{\beta}
\Big)
\Bigg) \\
&= -\, \frac{ A \, \epsilon^{\alpha}}{\varepsilon} \left(  \, \min a \, C_3^+\delta^\alpha - \max a \, \frac{C_2^+ \, \delta (1-\delta)^\alpha}{1 - \delta} \right) +\, \frac{ A \, \epsilon^{\beta}}{\varepsilon} \left(  \, \max a \,C_1^+\delta^{\beta} - \min a  \,\frac{C_4^+ \, \delta (1-\delta)^\beta}{1 - \delta} \right)\\
& \sim -\, \frac{ A \, \epsilon^{\alpha}}{\varepsilon} \left(  \, \min a \, C_3^+\delta^\alpha - \max a \, \frac{C_2^+ \, \delta (1-\delta)^\alpha}{1 - \delta} \right):= -\, \frac{ A \, \epsilon^{\alpha}}{\varepsilon} K(\alpha,\delta).
\end{align*}

Finally, as $\min a > 0$, one can prove that there exists $\delta > 0$, small enough such that $K(\alpha,\delta) > 0$. Passing to the limit as $\varepsilon \to 0$, we obtain the desired result.
\end{proof}

\begin{rem}\label{rem:2.10}
    Let us consider a measure $\mu = f dx$ with  $f \in L^1(\Pi)$ and such that, for $\epsilon$ small enough
    \begin{equation*}
f(x) = O(|x|^{-\beta})\quad \text{on} \,\,\, B(0, \epsilon), \text{ with } 0<\beta<1.
    \end{equation*}
    
    Then 
    \begin{equation*}
    \mu\bigl(B(0,r)\bigr) 
    \sim \int_{-r}^{\,r} |x|^{-\beta}\,dx
    = \frac{2\, r^{\,1-\beta}}{\,1-\beta\,}.
\end{equation*}

    It follows that $\dim_{\mathrm{loc}}(\mu,0) = 1 - \beta \;<\; 1.$ Moreover taking $y_0$ such that $f(y_0) > 0$ yield us to $\mu\bigl(B(y_{0},r)\bigr) \sim 2 f(y_{0})\, r$ and thus to $\dim_{\mathrm{loc}}(\mu,y_0) = 1$. Applying the previous corollary, we obtain that $I(\mu)$ diverges. We have thus constructed a probability measure that is absolutely continuous with respect to the Lebesgue measure, for which $I(\mu) = + \infty$.
\end{rem}

\begin{cor}\label{cor:2.11}
Let $\mu \in P(\Pi)$ be a probability measure and suppose there exists an open interval $I$ such that 
$\mu(I)=0$. Assume furthermore that there is a point $x_{0}$ for which one of the following four conditions holds:
\begin{enumerate}
    \item $-\infty < \alpha_{l}^{-} < \alpha_{r}^{+}$ and $\alpha_{l}^{-} < 1$;
    \item $-\infty < \alpha_{l}^{-} = \alpha_{r}^{+} =:\alpha < 1$;
    \item $-\infty = \alpha_{l}^{-} < \alpha_{r}^{+} \leq 1$;
\end{enumerate}

Then $I_{1}(\mu)=+\infty$.
\end{cor}

\begin{proof}
We take $y_0 \in I$, then it is clear that  $ \dim_{\text{loc}}^{-,r}(\mu, y_0) = dim_{\text{loc}}^{+,l}(\mu, y_0) = - \infty$. Considering the same test function as previously, we get
\begin{align*}
    I_1(\mu) &\ge -\frac{A}{\varepsilon}\Bigg( \min a \, C_3^+ (\epsilon \delta)^{\alpha^-_l}- \max a \,\frac{C_2^+ \, \delta}{1-\delta}\,(\varepsilon(1-\delta))^{\alpha^+_r}\Bigg) := -\, \frac{A}{\varepsilon} K(\alpha^-_l,\alpha^+_r,\delta).
\end{align*}

If we suppose that $-\infty < \alpha_{l}^{-} < 1$ or $ -\infty < \alpha_{l}^{-} = \alpha_{r}^{+} =: \alpha < 1$ then, as $\min a > 0$, one can prove that there exists $\delta > 0$, small enough such that $K(\alpha^-_l,\alpha^+_r,\delta) > 0$. Passing to the limit as $\varepsilon \to 0$, we obtain the desired result. If $\alpha_{l}^{-} = - \infty$ we obtain
\begin{align*}
    I_1(\mu) &\ge \frac{A}{\varepsilon} \max a \,\frac{C_2^+ \, \delta}{1-\delta}\,(\varepsilon(1-\delta))^{\alpha^+_r}.
\end{align*}

Repeating the previous calculation replacing $A$ by $-A$ lead us to 
\begin{align*}
    I_1(\mu) &\ge -\frac{A}{\varepsilon} \min a \,\frac{C_2^+ \, \delta}{1-\delta}\,(\varepsilon(1-\delta))^{\alpha^+_r},
\end{align*}
which diverges, passing to the limit as $\varepsilon \to 0$ under the condition $\alpha^+_r < 1$. Finally if $\alpha^+_r = 1$, passing to the limit as $A \to -\infty$ leads us to the result.
\end{proof}

\begin{rem}\label{rem:2.12}
    By symmetry, one may interchange the local
dimensions on the left and on the right in the assumptions (i.e.\ swap
$\alpha_l^{\pm}$ with $\alpha_r^{\pm}$) while keeping the same conclusions.
\end{rem}

\begin{rem}\label{rem:2.13}
    The previous corollary allows us to obtain an explicit result of the asymptotic of $I(\mu)$ where $\mu$ is the Cantor measure and more generally for any measure constructed by a similar iterative process (middle-$\alpha$ Cantor measure). As in the Cantor case, the local dimension is attained at zero with $\dim_{\text{loc}}(\mu, 0) = \ln(2)/\ln(3) < 1$. Moreover, taking an interval $I \subset ( 1/3,2/3)$ we get $I(\mu) = +\infty$.
\end{rem}

\section{The case where {\color{red}{\normalfont \textit{b}}} is disjoint from 0}

\subsection{ Asymptotic of the {\color{red}principal} eigenvalue}

This section is devoted to the asymptotic analysis of the {\color{red}principal} eigenvalue of the singularly perturbed spectral problem \eqref{2} under the hypothesis (H2).
However, we shall assume that \( a, b, c \in C^{0}(\Pi) \).

\begin{prop}\label{prop:3.1}
   Under the assumptions stated above,
     \begin{equation*}
     \label{18}
\lim\limits_{\epsilon \rightarrow 0} \lambda_\epsilon = \left(\int_{0}^{1} \frac{c(s)}{b(s)}ds\right) / \left(\int_{0}^{1} \frac{1}{b(s)}ds\right).
\end{equation*}
\end{prop}
    
\vspace{\baselineskip}

To prove this proposition, we will use the following Min-Max type formula \cite{Evans2010} (Courant minimax principle, page 347) to get a lower and an upper bounds of the {\color{red}principal} eigenvalue. The proof of the next lemma and its right part also provides a simple proof of the Protter-Weinberger theorem \cite{ProtterWeinberger1997} on $\Pi$ in our case.

\begin{lem}\label{lem:3.2}

Let us pose
\begin{equation*}
\textit{L}(u) := a(x) u'' + b(x)u' + c(x)u = \lambda u
\end{equation*}
where $a$, $b$, $c$ denote real-valued continuous functions defined on $\Pi$, and $\lambda$ denotes the {\color{red}principal} eigenvalue. Then the following holds
\begin{equation}
\label{6}
    \sup_{v>0}
    \inf_{x \, \in \, \Pi}  \frac{L v(x)}{v(x)} = \lambda = \inf_{v>0}
    \sup_{x \, \in \, \Pi} \frac{L v(x)}{v(x)},
\end{equation}
where the infimum and the supremum are taken on the twice differentiable positive real functions defined on $\Pi$.
    
\end{lem}

\begin{proof}

     Let us fix $v$, a twice differentiable positive real function defined on $\Pi$. By the Krein-Rutman theorem \cite{Brezis2011}, there exists a positive eigenfunction $\phi$ such that $\textit{L}(\phi) = \lambda \phi$, and therefore we can consider the positive ratio $\omega:= v/ \phi$. Setting $m := \inf_{x \in  \Pi}  \frac{L v(x)}{v(x)} $, we have
    \begin{equation*}
           0 \leq Lv - mv.
    \end{equation*}

Moreover, injecting $v = \omega \, \phi $ in the previous inequality, we derive
\begin{align*}
    Lv - mv & = a(x) (\omega''\phi + 2\omega'\phi' + \omega\phi'') + b(x)(\omega' \phi +\omega \phi') + c(x)(\omega \phi) -m(\omega \phi)\\
    & = \phi ( \widetilde L\omega + \lambda \omega - m \omega ),
\end{align*}
    where $\widetilde L \omega = a(x)\omega'' + (2a(x)\phi'/\phi + b(x))\omega'$.
    
\vspace{\baselineskip}  

Thus, using the positivity of  $\phi$ we get
\begin{equation*}
    0 \leq \widetilde L\omega + \lambda \omega - m \omega.
\end{equation*}
Taking the previous inequality at the point $x_0$, a maximum point of the twice differentiable function $\omega$, and using $\omega^\prime(x_0) = 0$ we get 

\begin{equation*}
    0 \leq a(x_0)\omega''(x_0) + \lambda \omega(x_0) - m \omega(x_0) \leq \lambda \omega(x_0) - m \omega(x_0),
\end{equation*}   
and $m \leq \lambda$ follows. Passing to the supremum, as the previous inequality is independent of our choice of $v$, we get one inequality. By proceeding in the same way, setting $M:= \sup_{x\in\Pi} \frac{L v(x)}{v(x)}$ and then taking the infimum over all admissible \(v > 0\), we obtain {\color{red} the other} inequality. Thus we have established
\begin{equation*}
    \sup_{v>0}
    \inf_{x \, \in \, \Pi}  \frac{L v(x)}{v(x)} \leq \lambda \leq \inf_{v>0}
    \sup_{x \, \in \, \Pi} \frac{L v(x)}{v(x)}.
\end{equation*}

The reverse inequalities are trivial. Indeed, for $v_0$ an eigenfunction associated with the {\color{red}principal} eigenvalue {\color{red} $\lambda$}, we have
\begin{equation*}
     \inf_{v>0}
    \sup_{x \, \in \, \Pi} \frac{L v(x)}{v(x)} \leq \sup_{x \, \in \, \Pi} \frac{L v_0(x)}{v_0(x)} = \lambda = \inf_{x \, \in \, \Pi} \frac{L v_0(x)}{v_0(x)} \leq \sup_{v>0}
    \inf_{x \, \in \, \Pi} \frac{L v(x)}{v(x)},
\end{equation*}
which gives us the result.
\end{proof}
 
\begin{proof}[Proof of Proposition \ref{prop:3.1}]

 We choose a test function $v_0 = exp(W_0)$ where $W_0$ is a solution of 
\begin{equation*}
    W_0' = (\lambda_0 -c(x))/ b(x) \; \; \; \text{with} \; \; \; \lambda_0 := \left(\int_{0}^{1} \frac{c(s)}{b(s)}ds\right) / \left(\int_{0}^{1} \frac{1}{b(s)}ds\right).
\end{equation*}

To get the result, we inject $v_0$ as a test function in \eqref{6} applied to our operator $\textit{L}_\epsilon$. The calculation gives us
\begin{align*}
    \frac{\textit{L}_\epsilon v_0(x)}{v_0(x)} &= \epsilon \, a(x) ((W_0')^2 + W_0'') + bW_0' + c
    = \epsilon \, a(x) ((W_0')^2 + W_0'') + \lambda_0.
\end{align*}

Finally, applying Lemma \ref{lem:3.2}, we get
\begin{align*}
     \lambda_0 + \epsilon \inf_{x\, \in \, \Pi}  a(x) ((W_0')^2 + W_0'') \leq \lambda_\epsilon \leq 
    \lambda_0 + \epsilon \sup_{x\, \in\, \Pi} a(x) ((W_0')^2 + W_0'').
\end{align*}

Passing to the limit as $\epsilon$ tends to $0$ gives us the result.
\end{proof}
\subsection{Explicit calculation of J}

Under hypothesis (H2), using Proposition \ref{prop:3.1}, we can rewrite $\Lambda(c)$ in the form
\begin{equation}\label{7}
\Lambda(c) := \lim\limits_{\epsilon \rightarrow 0} \lambda_\epsilon = \lambda_0 =  \int_0^1 c(s) \,\omega(s)\, ds \; \; \; \text{where} \; \; \; \omega(s) = \frac{1}{b(s)} \, \biggl / \int_0^1 \frac{1}{b(u)}du.
\end{equation}

We aim to establish the following result.

\begin{prop}\label{prop:3.3}
Under the assumptions stated above,
\begin{equation*}
J(\mu) = \begin{cases}
\;\;\; 0 & \displaystyle if \;\;\; \mu=\omega \, ds,\\
+\infty & \text{otherwise}.
\end{cases}
\end{equation*}
\end{prop}

\begin{proof}
    Using \eqref{7}, we get 
\begin{equation*}
   J(\mu) = \sup_{c\, \in \, C^0(\Pi)}\int_0^1 c(s) (d\mu(s)-\omega(s)ds\bigr). 
\end{equation*}

This supremum is finite if and only if the signed measure $\mu - \omega \, ds$ is zero. Otherwise, we can make it arbitrarily large. Indeed, if $\mu \neq \omega \, ds$, by Riesz–Markov theorem, there exists a continuous function $\phi$ on $\Pi$ such that $\int_0^1 \phi(s) (d\mu(s)-\omega(s)ds\bigr) \neq 0$. Considering $ M\phi$ or $-M\phi$ and letting $M$ go to $+ \infty$, we get the desired result.
\end{proof}

\subsection{Explicit calculation of I}

We are now interested in the asymptotic behavior of $I_\epsilon(\mu)$. Using the change of variable $u= exp(w)$ where $w$ is a $C^2(\Pi)$ function \eqref{5} becomes 
\begin{align}
\label{8}
I_\epsilon(\mu) &= \sup\limits_{w \in C^2(\Pi) } - \int_{\Pi} \epsilon \; a(x)( w''(x)+ w'(x)^2) + b(x)w'(x) \, d\mu(x) = \sup\limits_{v \in \Tilde{C}^{1,0}(\Pi)}\phi_{\mu,\epsilon}(v) ,
\end{align}
where $\Tilde{C}^{1,0}(\Pi) = \lbrace v \in C^1(\Pi), \int_{\Pi} v(x) dx= 0  \rbrace$ and $\phi_{\mu,\epsilon}(v) = - \int_{\Pi} \epsilon \, a\left( v'+ v^2\right) + bv\, d\mu.$

\vspace{\baselineskip}  

We state the main result of this section in the form of the following proposition.

\begin{prop}\label{prop:3.4}
Under the assumptions stated above,    
\begin{equation*}
I(\mu) = \begin{cases}
\;\;\; 0 & \displaystyle if \;\;\; \mu=\omega \, ds,\\
+\infty & \text{otherwise}.
\end{cases}
\end{equation*}
\end{prop}

\begin{proof}
First of all let us remark that for any $v \in \Tilde{C}^{1,0}(\Pi)$,
\begin{equation*}
    \phi_{\omega\, ds,\epsilon}(v)
     = - \, \epsilon \,C \int_{0}^{1} \,\frac{a(x)}{b(x)}\bigl(v'(x) + v(x)^2\bigr) \, dx \,\,\, \text{with} \,\,\, C = \left( \int_{0}^{1} \frac{1}{b(u)} \, du \right)^{-1}.
\end{equation*}

We will show that for any $\epsilon > 0$, the supremum in \eqref{8} with $\mu = \omega \, ds$ is attained at 
\begin{equation*}
    v_0 = \frac{\Tilde{a}'}{2\Tilde{a}} = \frac{1}{2} \,ln(\Tilde{a})' \; \; \; where \; \; \; \Tilde{a} = a/b.
\end{equation*}

Indeed, calculating we get
{\footnotesize
\begin{equation*}
    \phi_{\omega \, ds,\epsilon}(v_0)
     = -C \int_{0}^{1} \left( \epsilon \,\Tilde{a}(x)\bigl(v_0'(x) + v_0(x)^2\bigr) \right) \, dx
     = \frac{\epsilon \; C}{4} \int_0^1 \frac{(\Tilde{a}')^2}{\Tilde{a}} \; dx \,\,\, \text{with} \,\,\, C = \left( \int_{0}^{1} \frac{1}{b(u)} \, du \right)^{-1}.
\end{equation*}}

Moreover, one can show that the supremum is attained at this point, as for any $v \in \Tilde{C}^{1,0}(\Pi)$, the Young Inequality:
\begin{equation*}
    -\int_0^1 \Tilde{a}v' \, dx= \int_0^1 \Tilde{a}'v \, dx= \int_0^1 \frac{\Tilde{a}'}{\sqrt{2\Tilde{a}}}\, \times \sqrt{2\Tilde{a}} \;v \, dx \leq \frac{1}{2} \int_0^1 \frac{(\Tilde{a}')^2}{2\Tilde{a}}\, dx + \frac{1}{2} \int_0^1 2\Tilde{a}v^2 \, dx
\end{equation*}
gives us
\begin{equation*}
    -\int_0^1 \Tilde{a}v' +\Tilde{a}v^2 \, dx \leq \frac{1}{4} \int_0^1 \frac{(\Tilde{a}')^2}{\Tilde{a}} \, dx.
\end{equation*}

Passing to the limit as $\epsilon$ goes to $0$ in $\phi_{\omega,\epsilon}(v_0)$, we obtain that $I(\mu) = 0$ if $\mu=\omega \, ds$. We finish the proof showing that if $\mu$ is a measure different from $\omega \, ds$ then $ I(\mu) := \lim\limits_{\epsilon \rightarrow 0} I_\epsilon(\mu) = + \infty.$ Let us define the test function $v_\epsilon$ given by
\begin{equation*}
  v_\epsilon = \frac{1}{\sqrt[4]{\epsilon}} \, u \;  \;  \; \text{with} \;  \;  \; \int_0^1 u \, dx = 0.
\end{equation*}   

Then a small calculation gives us 
\begin{align*}
      &\phi_{\mu,\epsilon}(v_\epsilon) = - \int_{0}^1 \left(\epsilon \; a\left( v_\epsilon'+ v_\epsilon^2\right) + bv_\epsilon \right) \, d\mu \\
      &= - \int_0^1 \epsilon^{3/4} a(x) u'(x) \,
+ \sqrt\epsilon \, a(x) u(x)^2 \,
+ \frac{1}{\sqrt[4]{\epsilon}} \, b(x) u(x) \, d\mu(x)\\
&= -\, \frac{1}{\sqrt[4]{\epsilon}} \int_0^1 b(x) u(x) \, d\mu(x) + O(\sqrt\epsilon).    
\end{align*}

We end the proof by choosing $u$ such that $\int_0^1 b(x) u(x) \, d\mu(x)$ is negative and then passing to the limit as $\epsilon$ goes to $0$. To show that such a choice of $u$ is possible, it suffices to rule out that \(\int_0^1 b(x)\,u(x)\,d\mu(x)=0\) for every zero-mean functions \(u\). Suppose that it holds, then $ b\,\mu$ is proportional to $dx$. Since the case where \(b\,\mu\) is proportional to \(dx\) has been excluded by assumption, the result follows passing to the limit as $\epsilon$ goes to $0$.
\end{proof}

\section{The case where  {\normalfont \textit{b}}  has finitely many zeros}\label{sec:4}

\subsection{ Asymptotic of the {\color{red}principal} eigenvalue}\label{sec:4.1}

In this section, we are interested in the asymptotic analysis of the {\color{red}principal} eigenpair for the singularly perturbed
spectral problem \eqref{2} under the hypothesis (H3). We shall assume that $a \in C^2(\Pi)$, $b \in C^1(\Pi)$, $c \in C^0(\Pi)$ and that $b$ changes sign at least once on $\Pi$ and vanishes at a finite number of points 
\( \{ x_1, \dots, x_n \} \subset \Pi \). Each zero \(x_i\) of \(b\) is either of finite order $m_i$ or flat: 

\begin{itemize}
  \item Zero of finite order \(m_i \ge 1\).
  We say that \(x_i\) is a zero of finite order \(m_i\) if \(b \in C^{m_i}\) in a neighborhood of \(x_i\) and
  \[
    b(x_i)=b'(x_i)=\cdots=b^{(m_i-1)}(x_i)=0,
    \qquad b^{(m_i)}(x_i)\neq 0.
  \]
  \item Flat zero at \(x_i\).
  We say that \(x_i\) is a flat zero if \(b \in C^\infty\) in a neighborhood of \(x_i\) and
  \[
    b^{(k)}(x_i)=0 \quad \text{for all } k\in\mathbb{N}.
  \]
\end{itemize}
We now state the main result of this section in the form of the following theorem. 

\begin{thm}\label{thm:4.1}
   Under the assumptions stated above,
\begin{equation}
\label{9}
\lim\limits_{\epsilon \rightarrow 0} \lambda_\epsilon = \max\limits_{x, \, b(x)=0 } \left(-(0\lor b^\prime(x)) + c(x) \right).
\end{equation}  
\end{thm}

In subsection \ref{sec:4.1.1} we establish an asymptotic characterization of the
{\color{red}principal} eigenfunction; the main result
appears in Proposition \ref{prop:4.5}. Subsection  \ref{sec:4.1.2} provides a detailed
description and representation of the resulting asymptotics, culminating
in Proposition \ref{prop:4.8}. Through subsections  \ref{sec:4.1.3} and  \ref{sec:4.1.4}, we complete the proof of Theorem \ref{thm:4.1} by treating respectively the case of zeros of finite order and that of flat zeros. In subsection  \ref{sec:4.1.5}, we establish an exact formula for the asymptotics of the {\color{red}principal} eigenfunction under more restrictive assumptions.

\subsubsection{Asymptotic characterization of the
eigenfunctions}\label{sec:4.1.1}
 
\begin{defn} \label{defn:4.2} \cite{Fathi2008}, (Definition~7.2.3, page~215)
Let \( H : \mathbb{R} \times \Pi \to \mathbb{R} \) be a continuous function. A continuous function \( W :\Pi \to \mathbb{R} \) is called a \emph{viscosity solution} of the Hamilton–Jacobi equation
\begin{equation*}
    H(W'(x), x) = 0 \,\,\, \text{on} \,\,\, \Pi,
\end{equation*}
if the following conditions are satisfied:
\begin{itemize}
    \item For every test function \( \phi \in C^\infty(\Pi) \) such that \( W - \phi \) has a local maximum at a point \( \zeta \in \Pi \), we have
    \begin{equation*}
    H(\phi'(\zeta), \zeta) \leq 0.
    \end{equation*}
    \item For every test function \( \phi \in C^\infty(\Pi) \) such that \( W - \phi \) has a local minimum at a point \( \zeta \in \Pi \), we have
    \begin{equation*}
    H(\phi'(\zeta), \zeta) \geq 0.
    \end{equation*}
\end{itemize}
\end{defn}
\begin{rem}\label{rem:4.3}
    Suppose that $W \in C^{1}(\Pi)$ is a classical solution of $H (W'(x), x) = 0$ on $\Pi$, then it is also a viscosity solution of the equation. Indeed, let $\phi \in C^{\infty}(\Pi)$ be such that $ W - \phi$ has a local maximum or a minimum at $\zeta$, then $W'(\zeta) = \phi'(\zeta)$. It follows that the conditions in the previous definition hold, and so the notion of viscosity solution is weaker than the notion of classical solution. 
\end{rem}

\vspace{\baselineskip}

We now proceed with a reduction of problem \eqref{2}. 
By the Krein--Rutman theorem \cite{Brezis2011}, for every \( \epsilon > 0 \), the eigenfunction \( u_\epsilon \) can be chosen to be positive. 
Therefore, we may set
\begin{equation*}
    u_\epsilon(x) = \exp\left( -\frac{W_\epsilon(x)}{\epsilon} \right),
\end{equation*}
where \( W_\epsilon \) is a twice continuously differentiable function on \( \Pi \). Injecting this expression into \eqref{2} and setting 
\[
H(p,x) := a(x)p^2 - b(x)p,
\]
we obtain that the {\color{red}principal} eigenvalue \( \lambda_\epsilon \) satisfies
\begin{equation}
\label{10}
- a(x) W_\epsilon''(x) + \frac{1}{\epsilon} \, H\bigl( W_\epsilon'(x), x \bigr) + c(x) = \lambda_\epsilon.
\end{equation}

\begin{rem}\label{rem:4.4}
For any constant \( C \), if \( W_\epsilon \) satisfies \eqref{10}, then \( W_\epsilon + C \) also satisfies \eqref{10}. 
Therefore, without loss of generality, in the following, we restrict our analysis to the family of eigenfunctions normalized by the condition
\[
\int_0^1 W_\epsilon(x)\, dx = 0.
\]
\end{rem}

The next step is given by the following proposition. The argument is essentially the one of \cite{PiatnitskiRybalko2015} (Lemmas 3 and 4, pages 257-258), adapted in the case of $\Pi$.

\begin{prop}\label{prop:4.5}\renewcommand{\labelenumi}{\textit{\arabic{enumi}.}}
Under the assumptions stated above, the following two statements hold.
\begin{enumerate}
             \item For any sequence \( \epsilon \to 0 \), there exists a subsequence such that \( \lambda_{\epsilon} \) converges to a limit \( \lambda \) and such that the corresponding eigenfunctions \( W_{\epsilon} \) converge uniformly on \( \Pi \) to a limit function \( W \).
             \item In addition $W$ is a viscosity solution of $H(W'(x), x) = 0$ on $\Pi$.
\end{enumerate}
\end{prop}

\begin{proof}[Proof of 1. of Proposition \ref{prop:4.5}]
    We first show that $(\lambda_\epsilon)_\epsilon$ is bounded. Indeed, using the maximum and minimum points of the twice differentiable function $W_\epsilon$, we establish using \eqref{10} that for any $\epsilon > 0$, 
    \begin{equation}
    \label{11}
 \min\limits_{x \,\in\, \Pi} c(x) \leq \lambda_\epsilon \leq \max\limits_{x\, \in \,\Pi} c(x).
    \end{equation}

Now let us prove that, up to extracting a subsequence, $(W_\epsilon)_\epsilon$ converges uniformly on $\Pi$ to a limit $W$. Using the Ascoli Arzela theorem, it's enough to show that the family $(W_\epsilon)_\epsilon$ is uniformly bounded in $W_{1,\infty}(\Pi)$. As $\Pi$ is compact, it suffices to establish that $\mid\mid W'_\epsilon \mid\mid_\infty$ is uniformly bounded for $\epsilon \leq 1$.

\vspace{\baselineskip}

Fix any \(\epsilon \in (0,1]\) and let \(W_\epsilon\) satisfy \eqref{10}, normalized by \(\int_0^1 W_\epsilon\,dx=0\).
If \(W_\epsilon\) is null, there is nothing to prove. Hence, in what follows, we assume that \(W_\epsilon\) is not identically zero.

\vspace{\baselineskip}

As $W_\epsilon$ is twice differentiable, there exists $\zeta \in \Pi$, a maximum point of $(W'_\epsilon)^2$, such that $W'_\epsilon(\zeta) \ne 0$ (if not $W'_\epsilon(\zeta) = 0$ and the function $W'_\epsilon$ is constant equal to 0, which contradicts our hypothesis).
Differentiating we get $W''_\epsilon(\zeta) = 0$. Finally multiplying \eqref{10} by $\epsilon$ and taking the equality at point $\zeta$, we get
\begin{align}
\label{12}
a(\zeta) (W'_\epsilon)^2(\zeta)
= \epsilon[\lambda_\epsilon - c(\zeta)] + b(\zeta)W'_\epsilon(\zeta)
\leq [\max\limits c(x) - \min\limits c(x)] + CW'_\epsilon(\zeta).
\end{align}

The uniform boundlessness of $ (W'_\epsilon)^2(\zeta) $ follows directly from the previous inequality. We conclude the proof of $\textit1.$ noting that $\mid\mid W'_\epsilon \mid\mid_\infty = \sqrt{(W'_\epsilon)^2(\zeta)}$.
\end{proof}

Before starting the proof of $\textit2.$ we set the following lemma. 

\begin{lem}\label{lem:4.6}
    Let $(\psi_n)_n$ be a sequence of continuous functions converging uniformly to $\psi$ on $\Pi$. Suppose that x $\in \Pi$ is a strict local maximum (respectively, strict local minimum) point of $\psi$, then there exists a sequence $(x_n)_n$ converging to $x$ such that for $n$ large enough, $x_n$ will be a local maximum (respectively, local minimum) point of $\psi_n$.
\end{lem}

\begin{proof}
First of all, note that the result for a strict local minimum point follows directly from the result of the strict local maximum point, just by changing the sign of $(\psi_n)_n$ and $\psi$. Therefore, we will establish the result in the case of the strict local maximum point. 

\vspace{\baselineskip}

Without loss of generality, we assume that $0$ is a strict local maximum point of $\psi$. Then, there exists $r>0$ small enough such that $\psi$ restricted to $B(0,r)$ attains its unique maximum at $0$. Moreover, since $(\psi_n)_n$ converges uniformly to $\psi$, for $n$ sufficiently large there exists a point $x_n \in B(0,r/2)$ at which $\psi_n$ attains a local maximum. Extracting a subsequence, we assume that $(x_n)_n$ converges to a point $x_0 \in B(0,r)$. Then by hypothesis
\begin{align*}
\mid \psi_n(x_n) - \psi(x_0)\mid
&\leq \, \mid \psi_n(x_n) - \psi(x_n)\mid + \mid \psi(x_n) - \psi(x_0)\mid \, \rightarrow 0.
\end{align*}

Passing to the limit in $ \psi_n(x_n) \ge \psi_n(0) $ we get $ \psi(x_0) \ge \psi(0) $. Finally, since $r$ has been chosen sufficiently small so that $0$ is the unique global maximizer of $\psi$ on $B(0,r)$, the result follows.
\end{proof}

\begin{proof}[Proof of 2. of Proposition~\ref{prop:4.5}]
Let $\phi \in C^{\infty}(\Pi)$ such that $ W - \phi$ has a local maximum at $\zeta$. First, suppose that it is a strict local maximum. It follows from Lemma \ref{lem:4.6} that for any $\epsilon > 0$ there exists $\zeta_\epsilon$ a maximum point of $W_\epsilon - \phi$ such that $(\zeta_\epsilon)_\epsilon$ converges to $\zeta$. Using $W_\epsilon'(\zeta_\epsilon) =  \phi'(\zeta_\epsilon)$ and the inequality $W_\epsilon''(\zeta_\epsilon) \leq\phi''(\zeta_\epsilon)$ we get
\begin{align*}
& -\epsilon  \, a(\zeta_\epsilon) \phi''(\zeta_\epsilon)  + H( \phi'(\zeta_\epsilon),\zeta_\epsilon) + \epsilon  \,c(\zeta_\epsilon)\\
&\leq -\epsilon \, a(\zeta_\epsilon) W_\epsilon''(\zeta_\epsilon) + H(W_\epsilon'(\zeta_\epsilon),\zeta_\epsilon) + \epsilon  \,c(\zeta_\epsilon)  \\
&= \epsilon \lambda_\epsilon.
\end{align*}

Passing to the limit as $\epsilon$ goes to $0$ in the last inequality and considering the subsequence extracted in $\textit1.$ we get $H(\phi'(\zeta),\zeta) \leq 0$. 

\vspace{\baselineskip}

Now suppose that the local maximum point $\zeta$ is no longer strict. Then, there exists $(\phi_n)_n \in C^\infty(\Pi)$ such that $W - \phi_n$ has a strict local maximum at $\zeta$ and such that $(\phi_n'(\zeta))_n$ converges to $\phi'(\zeta)$. Repeating the above argument we get $H(\phi_n'(\zeta),\zeta) \leq 0$ and then $H(\phi'(\zeta),\zeta) \leq 0$ passing to the limit using the continuity of $H$. If $W - \phi$ has a local minimum at $\zeta$, the same calculation yields the opposite inequality, which concludes the proof.
\end{proof}

\subsubsection{The representation formula}\label{sec:4.1.2}

Since we have established that \(W\) is a viscosity solution of \(H(W'(x),x)=0\) on \(\Pi\), we can refine its description through the following definition, propositions and lemmas.

\begin{defn}\label{defn:4.7}
For any $x,z\in\Pi$, we define the Mañé potential associated with the Hamiltonian $H$ by
\begin{equation*}
d_H(z,x) = \inf\limits_{t > 0, \, \eta\in S_z^x(t)} \int_0^t L(-\eta'(s),\eta(s))ds
\end{equation*}
with 
\begin{equation*}
L(v,z)=\sup_{p\,\in\,\mathbb R}\lbrace vp-H(p,z)\rbrace 
=\sup_{p\,\in\,\mathbb R}\lbrace(v+b(z))p-a(z)p^{2}\rbrace
=\frac{\left(v+b(z)\right)^{2}}{4\,a(z)}
\end{equation*}
\begin{equation}
\label{13}
    \text{and} \,\,\, S_z^x(t) = \left\{\eta \in C^1_p([0,t]), \eta(0) = z ; \, \eta(t) = x\right\},
\end{equation}
where $C^1_p([0,t])$ denotes the class of piecewise $C^1$ curves on $[0,t]$.
\end{defn}

\begin{prop}\label{prop:4.8}
    Under the assumptions stated above, $W$ admits the following representation:
\begin{equation}
\label{14}
    W(z) = \min\limits_{x, \,b(x)=0 } (d_H(z,x) + W(x))
\end{equation}
\end{prop}

\begin{proof}
Using \cite{Fathi2008} (Proposition 5.3.8, page 197; and Theorem 8.6.1, page 255)  since $b \in C^{1}(\Pi)$ and $H$ is convex and superlinear in the second variable, the following general representation of $W$ holds:
\begin{equation}\label{15}
        W(z) = \inf\limits_{x \, \in \, A_H } \lbrace d_H(z,x) + W(x)\rbrace
\end{equation}
    where $A_H$ is called the Aubry set and is defined by
\begin{equation}\label{16}
    A_H := \biggl\lbrace x, \,\,\, \text{for  all} \,\,\,h > 0 \inf\limits_{T \ge h, \, \eta \, \in \, S_x^x(T)} \int_0^T L(-\eta'(s) ,\eta(s)) \, ds = 0
\biggl\rbrace. 
\end{equation}
Moreover, in the case where $b$ has finitely many zeros and changes sign on $\Pi$, one can show that $A_H$ consists exactly of the zeros of $b$, which will end the proof of the proposition. 

\vspace{\baselineskip}

We begin by showing that every zero of \(b\) belongs to the Aubry set. 
Indeed, let \(x\) be such that \(b(x)=0\). Consider the trajectory \(\gamma\) satisfying \(\gamma'=b(\gamma)\) with initial condition \(\gamma(0)=x\).
Since \(x\) is a zero of the field $b$, the solution is stationary, \(\gamma(t)\equiv x\); hence \(\gamma \in S_x^x(h)\) for any $h > 0$ and
\begin{equation*}
    \inf\limits_{ \eta \, \in \, S_x^x(T), \,T \ge h} \int_0^T L(-\eta'(s) ,\eta(s)) \, ds = \int_0^h L\bigl(-\gamma'(s),\gamma(s)\bigr)\,ds \;=\; 0.
\end{equation*}

To complete the proof, we need to show that for any point \(x\) with \(b(x)\neq 0\), \(x\) does not belong to the Aubry set.
We proceed by contradiction, assuming that for every \(h>0\),
\begin{equation*}
    \inf\limits_{ \eta \, \in \, S_x^x(T), \,T \ge h} \int_0^T L(-\eta'(s) ,\eta(s)) \, ds = 0.
\end{equation*}
Therefore there exist curves \(\gamma_n:[0,T_n]\to\Pi\) with \(T_n\to+\infty\) such that
\begin{equation}\label{17}
\lim_{n\to\infty}\int_0^{T_n} \frac{\left(-\gamma_n'+b(\gamma_n)\right)^{2}}{4\,a(\gamma_n)}\,ds=0.
\end{equation}

For any \(T>0\), since \(\Pi\) is compact and \(b\) is bounded, \eqref{17} implies that the restrictions
\(\gamma_n|_{[0,T]}\) are bounded in \(W^{1,2}([0,T])\). Hence, up to a subsequence,
\[
\gamma_n \xrightarrow{\text{uniformly}} \gamma \quad\text{in } C^{0}([0,T]) \,\,\, \text{and }\,\,\, 
\gamma'_n \rightharpoonup \gamma' \quad\text{in } L^{2}([0,T]).
\]
Moreover, passing to the limit in \eqref{17} on \([0,T]\) yields \(\gamma'=b(\gamma)\) a.e. on \([0,T]\). Indeed, expanding \eqref{17} and using the preceding convergences, we obtain 
\begin{align*}
0 &= \lim_{n\to\infty}
\int_0^{T}\frac{(\gamma_n')^2}{4\,a(\gamma_n)}
-\frac{\gamma_n' b(\gamma_n)}{2\,a(\gamma_n)}
+\frac{b(\gamma_n)^2}{4\,a(\gamma_n)}\,ds \\
&\ge \int_0^{T}\frac{(\gamma')^2}{4\,a(\gamma)}
-\frac{\gamma' b(\gamma)}{2\,a(\gamma)}
+\frac{b(\gamma)^2}{4\,a(\gamma)}\,ds = \int_0^{T} \frac{\left(-\gamma'+b(\gamma)\right)^{2}}{4\,a(\gamma)}\,ds \ge 0.
\end{align*}
Moreover, since $\gamma'\in L^{2}([0,T])$, the limit curve $\gamma$ is absolutely continuous and satisfies
\begin{equation*}
    \gamma(t)=\gamma(0)+\int_{0}^{t}\gamma'(s)\,ds
= x + \int_{0}^{t} b(\gamma(s))\,ds \,\,\, \text{for all } \,\,\, t\in[0,T].
\end{equation*}
Thus $\gamma$ is the solution of $\gamma'=b(\gamma)$ on $[0,T]$ with initial condition $\gamma(0)=x$.

\vspace{\baselineskip}

It follows from the dynamics of the system and the uniform convergence of
$(\gamma_n)_n$ to $\gamma$ that there exist a zero $y$ of $b$ and a sufficiently
small $\varepsilon>0$ such that one can find a time $T^{\varepsilon}>0$ and a
rank $N_{\varepsilon}$ for which $
\gamma_n(T^{\varepsilon}) \in B(y,\varepsilon)$
for all $n \ge N_{\varepsilon}$.

\vspace{\baselineskip}

Since $\gamma_n(T_n)=x$, each curve $\gamma_n$ must connect a point in $B(y,\varepsilon)$ to $x$. As there are only two possible directions, there exist $\alpha<\beta$, independent of $n$, such that, up to the extraction of a subsequence, for every $n > N$ one can find times $t_n^1<t_n^2$ such that for all $t\in[t_n^1,t_n^2]$, $\gamma_n(t)\in[\alpha,\beta]$ with $\alpha = \gamma_n(t^1_n)$ and $\beta = \gamma_n(t^2_n)$ and $\operatorname{sign}(\gamma_n'(t))\,\operatorname{sign}(b(\gamma(t)))=-1$. Moreover, one can suppose that $ 0 \leq \delta_1 < \mid b(x) \, \mid \leq \, \delta_2$ for all $x\in[\alpha,\beta]$.

\vspace{\baselineskip}

To prove this, we first define
\[
t_n^\varepsilon := \sup\{\, t \in [0, T_n], \gamma_n(t) \in B(y,\varepsilon) \,\}
\]
and note that either \(\gamma_n(t_n^\varepsilon) = y - \varepsilon\) or
\(\gamma_n(t_n^\varepsilon) = y + \varepsilon\).
Without loss of generality, we assume that all trajectories fall into the latter case and that $y + \varepsilon < x.$ Moreover, we may assume that the trajectories \(\gamma_n\) are strictly increasing on
\([t_n^\varepsilon, T_n]\).
Indeed, if this is not the case, it suffices to observe that one can define a sequence
\(\gamma_n^\ast\) of minimizing trajectories by setting
$m_n(t) := \max_{0 \le s \le t} \gamma_n(s)$ and removing the intervals on which \(m_n\) is constant, reparametrizing time accordingly.
Finally, from the dynamics of the system, there exists an interval \([\alpha,\beta]\)
such that \(b\) is negative on \([\alpha,\beta]\).
Defining \(t_n^1 := \gamma_n^{-1}(\alpha)\) and \(t_n^2 := \gamma_n^{-1}(\beta)\)
yields the result.

\vspace{\baselineskip}

Suppose that, up to a subsequence $\tau_n := t_n^2 - t_n^1$ converges to $0$, we obtain a contradiction as
\begin{equation*}
    \beta -\alpha = \int_{t^1_n}^{t^2_n} \gamma_n'(s)-b(\gamma_n(s)) +b(\gamma_n(s)) ds \leq  \left( \int_{t_n^1}^{t_n^2} (\gamma_n'(s)-b(\gamma_n(s)) )^2 ds \right)^{\frac12}\sqrt{\tau_n}+\delta_2\,\tau_n \rightarrow 0.
\end{equation*}
Assuming instead that the sequence $(\tau_n)_n$ converges to some $\tau>0$, or that it diverges, also yields a contradiction. Indeed, passing to the limit in
\begin{equation*}
\int_{t_n^1}^{t_n^2} \frac{\left(-\gamma_n'(s)+b(\gamma_n(s))\right)^{2}}{4\,a(\gamma_n(s))}\,ds \ge \int_{t_n^1}^{t_n^2} \frac{b(\gamma_n(s))^{2}}{4\,a(\gamma_n(s))}\,ds \ge \frac{\delta_1 \, \tau_n}{4 \,\max a}
\end{equation*}
yields $0 \ge (\delta_1 \, \tau)/(4 \,\max a) > 0$ from which a contradiction follows.
\end{proof}

\begin{rem}\label{rem:4.9}
Since a piecewise differentiable curve from $x$ to $z$ can be written as the concatenation of piecewise differentiable curves from $x$ to $y$ and from $y$ to $z$, $d_H(\cdot,\cdot)$ verifies the triangular inequality. 
\end{rem}

\begin{rem}\label{rem:4.10}
It follows from this positivity of $L$ and from the fact that $b$ changes sign on $\Pi$ that $d_H(x,y)+d_H(y,x) > 0$ holds if and only if $x$ is different from $y$. 

\vspace{\baselineskip}

Indeed suppose that  $d_H(x,y) = d_H(y,x) = 0$. Then, there exist sequences \(t_n>0\), \(s_n>0\) and piecewise \(C^1\) curves
\(\eta^{x,y}_n \in S_x^y(t_n)\) and \(\eta^{y,x}_n \in S_y^x(s_n)\) forming minimizing sequences for \(d_H(x,y)\) and \(d_H(y,x)\), i.e
\begin{equation*}
\lim_{n\to\infty}\int_0^{t_n} L \bigl(-\eta'^{x,y}_n(s),\eta^{x,y}_n(s)\bigr)\,ds \;=\; 0 \,\,\, \text{and} \,\,\,
\lim_{n\to\infty}\int_0^{s_n} L \bigl(-\eta'^{y,x}_n(s),\eta^{y,x}_n(s)\bigr)\,ds \;=\; 0.
\end{equation*}

Define the piecewise \(C^{1}\) curves \(\gamma_n:[0,t_n+s_n]\to\Pi\) by concatenating \(\eta^{x,y}_n\) and \(\eta^{y,x}_n\). It is clear that $\gamma_n \in S_x^x(t_n+s_n)$ and $d_H(x,x) = 0$ as
\begin{equation*}
0 \leq d_H(x,x) \leq \lim_{n\to\infty}\int_{0}^{t_n+s_n} L\!\bigl(-\gamma_n',\gamma_n\bigr)
= d_H(x,y)+d_H(y,x) = 0.
\end{equation*}

Moreover, by taking the \(k\)-fold concatenation of \(\gamma_n\) we obtain
curves \(\eta_{n,k}\in S_x^x(k\,(t_n+s_n))\) such that 
\[
\lim_{n\to\infty} \int_0^{k\,(t_n+s_n)} L\bigl(-\eta'_{n,k}(s),\eta_{n,k}(s)\bigr)\,ds
= k \, d_H(x,x) = 0.
\]
 Since for any $h > 0$, we can take $k$ large enough such that $h < k\,(t_n+s_n)$ we get that $x \in A_H$. This also implies that $y\in A_H$ and more generally, any point that is visited infinitely many times by the sequence of curves \(\gamma_n\) also belongs to \(A_H\). As on \(\Pi\), there are only two possible directions between two distinct points. It follows from the preceding construction that \(A_H\) contains at least one nontrivial interval, which provides us a contradiction, since under our assumptions \(A_H\) consists of finitely many points.

\end{rem}

The following lemma refines the representation by removing the minimum in \eqref{14}. The proof follows the strategy in \cite{PiatnitskiRybalko2015} (Step 1, pages 258-259).

\begin{lem}\label{lem:4.11}
    There exists $x_0$ such that $b(x_0) = 0$ and $U$ a neighborhood of $x_0$ such that $ W(x) = d_H(x,x_0) + W(x_0)$ for all $x \in U$.
\end{lem}

\begin{proof}
    We define the relation
\(\preceq\) on $\{ x_1, \dots, x_n \}$, the set of the zeros of $b$, by
\[
x \preceq y \quad \Longleftrightarrow \quad W(y)=d_H(y,x)+W(x).
\]
    The previous remarks allow us to prove that the relation \(\preceq\) defines a partial order. The reflexivity is quickly checked as $d_H(x,x) =0$. The transitivity follows from Remark \ref{rem:4.9}, as $d_H(\cdot,\cdot)$ verifies the triangle inequality. Finally, the antisymetry is true due to Remark \ref{rem:4.10}. as $x \preceq y$ and $y \preceq x$ implies $d_H(x,y)+d_H(y,x) = 0$. 
    
\vspace{\baselineskip}    
    
    As there exists at least one minimal element associated with that partial order relation, we pick up one, and we note it $x_0$. By definition, for any $y$ such that $b(y)=0$, either $x_0 \preceq y$, or $x_0$ and $y$ are not comparable. Finally, we remark that we have constructed our candidate.
    
\vspace{\baselineskip}

Indeed, suppose that there exists no neighborhood of $x_0$ such that $W(x) = d_H(x,x_0) + W(x_0)$ in that neighborhood. Then, there exist a sequence $(x_n)_n$ converging to $x_0$ and a sequence of elements $X_n \in \lbrace x, b(x)=0\rbrace \setminus \lbrace x_0\rbrace$, such that for each $n$, $W(x_n) = d_H(x_n,X_n) + W(X_n)$. However, up to extracting a subsequence using the fact that the set $\lbrace x, b(x)=0\rbrace \setminus \lbrace x_0\rbrace$ is finite, we can construct a sequence $(x_n)_n$ converging to $x_0$ and an element $X$ different from $x_0$ such that for each $n$, $W(x_n) = d_H(x_n, X) + W(X)$. Passing to the limit, using the continuity of $d_H$, we get $W(x_0) = d_H(x_0, X) + W(X)$. As we have constructed $X \neq x_0$ such that $ X \preceq x_0$, we get a contradiction.
\end{proof}

\begin{rem}\label{rem:4.12}
A zero \( x_0 \) of \( b \) for which the previous lemma applies will be referred to as a significant component. Moreover, whenever such a point is considered in what follows, we will, without loss of generality, up to adding a constant to \(W\) and a translation if necessary, always assume that \(W(x_0)=0\) and \(x_0=0\).
\end{rem}

\vspace{\baselineskip}

Before beginning the proof of the main theorem, we set the following lemma, which will be useful in its proof. By abuse of notation, we write \( \eta \in B(0,\delta) \) if, for all \( t \) for which \( \eta \) is defined, we have \( \eta(t) \in B(0,\delta) \).

\begin{lem}\label{lem:4.13}
Let \( x_0 \) be a zero of $b$. Then, for any sufficiently small \( \delta > 0 \), and for any point \( x \in B(x_0, \delta) \), the piecewise \( C^1 \) curves minimizing \( d_H(x, x_0) \) remain entirely within the ball \( B(x_0, \delta) \). In other words
\begin{equation*}
d_H(x,x_0) = \inf_{\substack{t \,>\, 0,\, \eta\, \in\, S_x^{x_0}(t) \\ \eta \, \in \, B(x_0, \delta)}} 
\int_0^t L(-\eta'(s), \eta(s))\, ds.
\end{equation*}
\end{lem}

\begin{proof}
 Without loss of generality, we assume that $x_0 = 0$. 
If $b$ is positive in a small left neighborhood of $0$ or negative in a small right neighborhood, then, by considering the trajectories that follow the vector field $b$, we obtain the existence of a sufficiently small $\delta > 0$ such that, for any $x \in B(x_0, \delta)$, the piecewise $C^1$ curves minimizing $d_H(x, x_0)$ remain entirely within $B(x_0, \delta)$, and satisfy $d_H(x, x_0) = 0$.

\vspace{\baselineskip}

We now suppose that $b$ is positive in a small right neighborhood of $0$, which corresponds to the case where the field is repulsive from this side. Let us set $y_0 > 0$ small enough such that $d_H(y_0,0) = c > 0$. Such a point should exist otherwise; if there exists $y_0$ small enough such that $d_H(y,0) = 0$, then as $d_H(0,y) = 0$ a contradiction follows. By continuity of $d_H$, for any $\delta \in (0, y_0)$ small enough, all $x \in (0,\delta]$ verifies $d_H(x,0) < c$ (it suffices to consider the curves following the field $-b$ and staying in $(0,\delta]$). Thus, for any $x \in (0,\delta]$, the trajectory minimizing $d_H(x, 0)$ can never leave the interval $(0, y_0]$, otherwise it will lead us to the contradiction $d_H(x, 0) \ge d_H(y_0, 0) = c$.

\vspace{\baselineskip}

We have proved that for any point $x \in (0,\delta]$, the sequence of curves minimizing $d_H(x, 0)$ never goes over the point $y_0$. Finally, it is clear that, in the limit, any curves of the minimizing sequence starting in $(0,\delta]$, stay in $(0,\delta]$, as they all reach $0$ by the right, and adding another path that makes them go out of $(0,\delta]$ should only increase its work over a longer trajectory.
\end{proof}

\subsubsection{The case of zeros of finite order}\label{sec:4.1.3}

The following subsection generalizes the arguments of Section 5 in \cite{Fathi2008} (pages 261-263) to the case of zeros of finite order, and yields us the limit arguing as in \cite{Fathi2008} (Step II and Step III, pages 259-261).

\begin{prop}\label{prop:4.14}
Let $x_0$ be a significant component which is a zero of finite order $m \ge 1$.  
Define $\sigma(x_0) := 0 \,\vee\, b'(x_0)$ and $
\Lambda(x_0) := -\sigma(x_0) + c(x_0).$
Then, for every sufficiently small $\delta > 0$, there exist a neighborhood $U_\delta$ of $x_0$ and functions $W_\delta^\pm \in C^2(U_\delta)$ such that
\begin{equation}
\label{18}
W_\delta^\pm(x_0) = 0,
\,\,\,\,\,
W_\delta^-(x) < W(x) - W(x_0) < W_\delta^+(x)
\,\,\,\,\, \text{for all } \,\,\,\,\, x \in U_\delta \setminus \{x_0\}.
\end{equation}
Moreover, for any sequence $(x_\varepsilon)_{\varepsilon>0}$ converging to $x_0$ as $\varepsilon \to 0$, we have
\begin{equation*} \liminf\limits_{\delta \rightarrow 0} \liminf \limits_{\epsilon \rightarrow 0} \left( -a(x_\epsilon) \, W_\delta^{+\,\prime\prime}(x_\epsilon) + \frac{1}{\epsilon} \, H\big(W_\delta^{+\,\prime}(x_\epsilon), x_\epsilon\big) + c(x_\epsilon) \right) \ge \Lambda(x_0), \end{equation*} \begin{equation*} \limsup_{\delta \to 0} \limsup \limits_{\epsilon \rightarrow 0} \left( -a(x_\epsilon) \, W_\delta^{-\,\prime\prime}(x_\epsilon) + \frac{1}{\epsilon} \, H\big(W_\delta^{-\,\prime}(x_\epsilon), x_\epsilon\big) + c(x_\epsilon) \right) \leq \Lambda(x_0). \end{equation*}
\end{prop}

\begin{proof}
To prove the proposition, it is sufficient to establish \eqref{18} and that for any $x$ in a small neighborhood of $x_0$, we have  
    \begin{equation}
    \label{19}
 H(W_\delta^{-\,\prime}(x),x) \leq 0 \leq H(W_\delta^{+\,\prime}(x),x) \;\;\; \text{and}
\end{equation}
    \begin{equation}
    \label{20}
\limsup\limits_{\delta \rightarrow 0} a(x_0)W_\delta^{-\,\prime\prime}(x_0) = \liminf\limits_{\delta \rightarrow 0} a(x_0)W_\delta^{+\,\prime\prime}(x_0)= -\sigma(x_0).
\end{equation}

Let \( x_0 \) be a significant component and a zero of \( b \) of order \( m \ge 1 \). We recall that we take \( x_0 = 0 \) and  \( W(x_0) = 0 \). We first construct $W_\delta^{-}$. We split the construction into two cases depending on the parity of $m$. First, we suppose that $m$ is odd and pose the ansatz
\begin{equation*}
    W_\delta^{-}(x) := (\Gamma - \delta D) x^{m+1} \,\,\, \text{with} \,\,\, \delta > 0.
\end{equation*}
    
To get $\eqref{19}$, we first inject our ansatz in $H$. Developing the coefficient near 0, we obtain
\begin{align*}
&H\left( W_\delta^{-\,\prime}(x), x \right)
= a(x)\left( (m+1)(\Gamma - \delta D) x^m \right)^2 - b(x)(m+1)(\Gamma - \delta D)x^m \\
&= \left(a(0) + O(x)\right)(m+1)^2 (\Gamma - \delta D)^2 x^{2m} - \left(b^{(m)}(0) x^m + O(x^{m+1})\right)(m+1)(\Gamma - \delta D) x^m \\
&= x^{2m} \left[ (m+1)^2 a(0)(\Gamma - \delta D)^2 - (m+1)b^{(m)}(0)(\Gamma - \delta D) \right] + O(x^{2m+1}).
\end{align*}

Expanding in powers of \( \delta \) and factoring out the leading term, we obtain
\begin{align*}
&(m+1)^2 a(0)(\Gamma - \delta D)^2 - (m+1)b^{(m)}(0)(\Gamma - \delta D)\\ 
&=  (m+1)^2 a(0)(\Gamma^2 - 2\Gamma \delta D + \delta^2 D^2) - (m+1)b^{(m)}(0)(\Gamma - \delta D) \\
&= \left((m+1)^2 a(0) \Gamma^2 - (m+1)b^{(m)}(0) \Gamma\right) - \delta \left( 2(m+1)^2 a(0)\Gamma D - (m+1)b^{(m)}(0) D \right)+ O(\delta^2).
\end{align*}

Choosing $\Gamma$ and $D$ such that
\begin{equation}
\label{21}
    \begin{cases}
(m+1)^2 a(0) \Gamma^2 - (m+1)b^{(m)}(0) \Gamma = 0, \\
2(m+1)^2 a(0)\Gamma D - (m+1)b^{(m)}(0) D = \alpha > 0,
\end{cases}
\end{equation}
ensures that for $\delta > 0$ small enough, the first inequality in \eqref{19} holds. Moreover solving the system for $\alpha = 1$, we get
\begin{equation*}
    \left\{
\begin{aligned}
&\text{either} && \Gamma = 0, \quad D = -\dfrac{1}{(m+1)b^{(m)}(0)}, \\
&\text{or}     && \Gamma = \dfrac{b^{(m)}(0)}{(m+1)a(0)}, \quad D = \dfrac{1}{(m+1)b^{(m)}(0)}.
\end{aligned}
\right.
\end{equation*}

Set for our constructed function $\Gamma:= max(0 \, , b^{(m)}(0)/(m+1)a(0))$ and let $ D $ be the corresponding choice such that equation \eqref{21} is satisfied with $ \alpha = 1 $. It follows from our choice that 
 \begin{equation*}
  \limsup\limits_{\delta \rightarrow 0} a(x_0)W_\delta^{-\,\prime\prime}(x_0)= \sigma(x_0).
\end{equation*}

We now show that $W_\delta^-(x) < W(x)$ holds in a punctured neighborhood of $0$. Let choose $\delta > 0$ small enough and consider a piecewise $C^1$ curve $\eta \in S_x^0(t)$ such that the trajectory never get out of the ball $B(0, \delta)$. Using the Fundamental Theorem of Calculus and the Fenchel-Young inequality, we get
\begin{align*}
    W_\delta^-(x)
    &= -\int_0^t W_\delta^{-\prime}(\eta(s))\, \eta'(s)\, ds = \int_0^t W_\delta^{-\prime}(\eta(s)) \cdot (-\eta'(s))\, ds \\
    &\leq \int_0^t L(-\eta'(s), \eta(s))\, ds + \int_0^t H(W_\delta^{-\prime}(\eta(s)), \eta(s))\, ds.
\end{align*}
Passing to the infimum, we obtain
\begin{align*}
W_\delta^-(x)  \leq \inf\limits_{t > 0,\,  \eta\,\in \,S_x^0(t)\cap B(0,\delta)} \left( \int_0^t L(-\eta'(s),\eta(s))ds + \int_0^t H(W_\delta^{-\prime}(\eta(s)),\eta(s))ds \right).
\end{align*}
Using that $H(W_\delta^{-\prime}(x),x) \leq -\delta x^{2m}$ holds for small enough $x$ and  Lemma {\color{red} \ref{lem:4.13}}, we get
\begin{equation*}
W_\delta^-(x)  \leq \inf\limits_{t > 0,\,  \eta \, \in \,S_x^0(t)\cap B(0,\delta)} \left( \int_0^t L(-\eta'(s),\eta(s))ds\right) = d_H(x,0) = W(x).
\end{equation*}

At first sight, it is not clear that one obtains the desired strict inequality. However replacing $\delta$ by $\delta/2$, the same calculation leads us to $W_{\delta/2}^-(x)\leq W(x)$. As $k+1$ is even, $W_\delta^-(x)< W_{\delta/2}^-(x)$ holds on a punctured neighborhood of $0$. Combining the two previous inequalities, we get the desired result.

\vspace{\baselineskip}

When $m$ is even, we cannot conclude this way. Without loss of generality, we assume that $b^{(m)}(0) >0$. In this case, for any $x\leq 0$
\begin{equation*}
W(x) = d_H(x,x_0) = \inf_{\substack{t \,>\, 0,\, \eta\, \in\, S_x^{x_0}(t) \\ \eta \, \in \, B(x_0, \delta)}} 
\int_0^t L(-\eta'(s), \eta(s))\, ds = 0
\end{equation*}
choosing the minimizing curves that follow the field $b$. It follows that $W_\delta^-(x) < 0 = W(x)$ for $x\leq 0$. To complete the inequality on $x> 0$ it suffices to remark that $W_\delta^-(x)< W_{\delta/2}^-(x)$ holds for any positive $x$.

\vspace{\baselineskip}

We now construct $W_\delta^{+}$. We begin with the case when $m=1$. We set
\begin{equation*}
    W_\delta^{+}(x) := (\Gamma + \delta D) x^{2} \,\,\, \text{with}\,\,\, \delta > 0.
\end{equation*}

We take $\Gamma:= max(0\,, b^{\prime}(0)/2a(0))$ and let $D$ be the corresponding choice such that equation \eqref{21} is satisfied with $ \alpha = 1$. With this choice, we obtain the second parts of \eqref{19} and \eqref{20} as
\begin{gather*}
H\bigl(W_\delta^{+\,\prime}(x), x\bigr)
= \delta x^2 + 4\delta^2\, a(0) D^2 x^2 + O(|x|^3)\\
\text{and}\,\,\, \limsup_{\delta \to 0} a(x_0)\, W_\delta^{-\,\prime\prime}(x_0)= \sigma(x_0).
\end{gather*}

We end the proof showing that $W(x) < W_\delta^+(x)$ holds in a punctured neighborhood of $0$. Using Proposition \ref{prop:4.8} {\color{red} and Lemma \ref{lem:4.11}}, we know that for every $t > 0$, and piecewise $C^1$ curves $\eta \in S_x^0(t)$
\begin{equation*}
    W(x) \leq  \int_0^t L(-\eta'(s),\eta(s))ds.
\end{equation*}

We consider a testing curve that is the solution of the equation 
\begin{equation*}
    \eta^\prime = K\eta \;\;\; \text{where} \;\;\; K =  -(4 a(0)\Gamma- b'(0)).
\end{equation*}

We remark that $K$ is always negative as
\begin{equation*}
    -(4a(0)\Gamma - b'(0)) =
\begin{cases}
- b'(0), & \text{if } 0 \leq \dfrac{b'(0)}{2a(0)}, \\
\phantom{-} b'(0), & \text{else }.
\end{cases}
\end{equation*}

Then $\eta(t) \rightarrow 0$ and $ \eta(t)^\prime \rightarrow 0$ as $ t \rightarrow +\infty$. Moreover, the decays are exponentially fast. So there exists a constant $C$, independent of $t$, such that $|\eta(t)| < C|x|$ and $|\eta^\prime(t)| < C|x|$. It follows that
\begin{align*}
W(x) & 
\leq\int_0^{\infty} L(-\eta^\prime(s),\eta(s))ds \\ & = \frac{1}{4} \int_0^{\infty}  a(\eta(s))(-\eta^\prime(s) + b(\eta(s)))^2 ds\\
    &= \frac{1}{4} \int_0^{\infty}  a(0)(-\eta^\prime(s) + b'(0)\eta(s)))^2 + O(|\eta(s)|^3)ds\\
    &= \int_0^{\infty}  4 \, a(0)\Gamma^2 \eta(s)^2 + O(|\eta(s)|^3)ds.
\end{align*}

Moreover noting that
\begin{equation*}
-2\Gamma \eta \eta' = 2\Gamma (4 a(0)\Gamma- b'(0))\eta^2 = 2(4 a(0)\Gamma^2- b'(0)\Gamma)\eta^2 = 4a(0)\Gamma^2\eta^2
\end{equation*}
we get 
\begin{equation*}
W(x) \leq \int_0^{\infty}  -2\Gamma \eta(s) \eta(s)'+ O(|\eta(s)|^3) \,ds = \Gamma x^2 + O(|x|^3) < \Gamma^{+}_\delta(x).
\end{equation*}

As the last inequality is strict except at 0, we get the desired result. 

 \vspace{\baselineskip} 
 
We finish the proof with the construction of $W_\delta^{+}$ when $m > 1$. We pose the ansatz
    \begin{equation*}
        W_\delta^{+}(x) := \delta x^2 \,\,\, with \,\,\, \delta > 0.
    \end{equation*}
    
As $0$ has order $m \ge 2$, we first verify that the second inequality in \eqref{19} holds as
\begin{equation*}
    H(2\delta x, x) = 4\delta^2 a(x) x^2 - 2\delta b(x) x = 4\delta^2 a(x) x^2 + O(x^{m+1}).
\end{equation*}

To prove the second inequality of \eqref{18}, let us consider the solution of the equation
\begin{equation}
\label{22}
    -\eta'(s) + \frac{b^{(m)}(0)}{m!} \, \eta(s)^m = \alpha \, \delta \, \eta(s) \,\,\, with \,\,\,\alpha > 0.
\end{equation}

If $b^{(m)}(0) < 0$, the equation admits a unique global solution 
$\eta(s)$ and it is well known that $\eta(t) \rightarrow 0$ and $ \eta(t)^\prime \rightarrow 0$ as $ t \rightarrow +\infty$ with exponentially fast decays. When $b^{(m)}(0) > 0$, one can choose a solution with the same behavior provided that the initial condition is small enough. So there exists a constant $C$, independent of $t$, such that $|\eta(t)| < C|x|$ and $|\eta^\prime(t)| < C|x|$. It follows that
\begin{align*}
W(x) & 
\leq\int_0^{\infty} L(-\eta^\prime(s),\eta(s))ds \\ & = \frac{1}{4} \int_0^{\infty}  a(\eta(s))(-\eta^\prime(s) + b(\eta(s)))^2 \,ds\\
    &= \frac{1}{4} \int_0^{\infty}  a(0)\left(-\eta'(s) + \frac{b^{(k)}(0)}{m!}\eta(s)^m \right)^2 + O(|\eta(s)|^{m+1}) \,ds\\
    &= \frac{1}{4} \int_0^{\infty}  a(0)\alpha^2 \delta^2\eta(s)^2 + O(|\eta(s)|^{m+1})\,ds.
\end{align*}

Moreover multiplying \eqref{22} by $\alpha \, \delta \eta $ we get
\begin{equation*}
    - \alpha \, \delta \, \eta(s) \, \eta'(s) + \alpha \, \delta \, \frac{b^{(m)}(0)}{m!} \, \eta(s)^{m+1} = \alpha^2 \, \delta^2 \, \eta(s)^2,
\end{equation*}
thus
\begin{align*}
    W(x)
\leq \frac{a(0)}{8}\int_0^{\infty} - 2 \, \alpha \,\delta \, \eta(s) \, \eta'(s) + O(|\eta(s)|^{m+1}) \,ds = \frac{a(0)}{8} \, \alpha \, \delta x^2 + O(|x|^{m+1}).
\end{align*}

Taking $\alpha = 4/a(0)$ we get
\begin{equation*}
    W(x) \leq \frac{1}{2}\delta x^2 + O(|x|^{m+1}) \leq \delta x^2 = W_\delta^{+}(x).
\end{equation*}

As the last inequality is strict except at 0, we get the desired result. 
\end{proof}

We now explain how Proposition \ref{prop:4.14} leads us to the result. It allows us to get the following characterization in terms of significant components of the asymptotic of our {\color{red}principal} eigenvalue, as well as the upper bound of Theorem \ref{thm:4.1}.

\begin{prop}\label{prop:4.15}
    Let $x_0$ be a significant component that is a zero of finite order $m \ge 1$, then 
    \begin{equation*}
        \lim\limits_{\epsilon \rightarrow 0} \lambda_\epsilon = \Lambda(x_0) = -(0\lor b^\prime(x_0))+c(x_0).
    \end{equation*}
\end{prop}

\begin{proof}
    Thanks to Proposition \ref{prop:4.14}, $W - W_\delta^- > W(x_0)$ on $U_\delta \backslash \lbrace x_0\rbrace$ and $W(x_0) - W_\delta^-(x_0)  = W(x_0)$. Hence, $x_0$ is a local minimum point of $W - W_\delta^-$. Moreover, from $\textit2.$ of Proposition \ref{prop:4.5}, it follows that $(W_\epsilon - W_\delta^-)_\epsilon$ converges uniformly to $W - W_\delta^-$. Using Lemma \ref{lem:4.6}, we get a sequence $(x_\epsilon)_\epsilon$ converging to $x_0$ such that for $\epsilon$ small enough, $x_\epsilon$ will be a minimum point of $W_\epsilon - W_\delta^-$. Finally using $ W_\epsilon'(x_\epsilon) = W_\delta^{-\prime}(x_\epsilon)$ and the inequality $W_\epsilon''(x_\epsilon) \leq W_\delta^{-\prime\prime}(x_\epsilon)$, we derive
\begin{align*}
    \lambda_\epsilon &= -a(x_\epsilon)W_\epsilon''(x_\epsilon) + \frac{1}{\epsilon}H( W_\epsilon'(x_\epsilon),x_\epsilon) + c(x_\epsilon) \\ &\leq -a(x_\epsilon) \, W_\delta^{-\prime\prime}(x_\epsilon) + \frac{1}{\epsilon}H(W_\delta'(x_\epsilon),x_\epsilon) + c(x_\epsilon).
\end{align*}

Passing to the limit using Proposition \ref{prop:4.14}, we get
\begin{equation*}
  \limsup\limits_{\epsilon \rightarrow 0} \lambda_\epsilon \leq  \Lambda(x_0).
\end{equation*}

Considering the previous development with $W_\delta^+$ instead of $W_\delta^-$ and working with the maximum point instead of the minimum, we derive
\begin{equation*}
    \Lambda(x_0) \leq  \liminf\limits_{\epsilon \rightarrow 0} \lambda_\epsilon.
\end{equation*}

Combining the two previous inequalities yields the desired result.
\end{proof}
\begin{rem}\label{rem:4.16}
    As noted above, this proposition yields the upper bound in the main asymptotics, since it implies that
    \begin{equation*}
          \limsup\limits_{\epsilon \rightarrow 0} \lambda_\epsilon \leq  \Lambda(x_0) \leq \max\limits_{x, \, b(x)=0 } \left(-(0\lor b^\prime(x)) + c(x) \right).
    \end{equation*}
\end{rem}

\begin{rem}\label{rem:4.17}
    The same argument and constructions extend to the situation in the {\color{red} non hyperbolic} case where the left and right $k$-th derivatives differ at $\{ b = 0 \}$ for any $k \ge 2$. Indeed, since the $k$-th derivatives vanish in the limit as $\delta \to 0$, the resulting behaviour is unchanged, and we obtain the same result.
\end{rem}

We now conclude this subsection by completing the proof of Theorem~\ref{thm:4.1}, establishing its lower bound in the case of zeros of finite order.

\begin{proof}[Proof of Theorem \ref{thm:4.1}]
Let us set the modified operator:
\begin{equation*}
\Tilde{L}_\epsilon(v_\epsilon) := L_\epsilon(v_\epsilon)-\frac{1}{\epsilon}k \, v_\epsilon = \epsilon \, a(x) v_\epsilon'' + b(x) v_\epsilon' + (c(x)-\frac{1}{\epsilon}k(x))\, v_\epsilon = \mu_\epsilon v_\epsilon,
\end{equation*}
where $y_0$ is a fixed zero of $b$ and $k$ is chosen as a non negative function $k \in C^{2}(\Pi)$ such that $k$ vanishes only in a small neighborhood $U$ of $y_0$. Using Protter–Weinberger variational theorem \cite{ProtterWeinberger1997} on $\Pi$ we get
\begin{align*}
\mu_\epsilon &= \inf\limits_{u>0}\sup\limits_{x\in\Pi}\frac{(\overset{\sim}{\textit{L}}_\epsilon u)(x)}{u(x)} \leq \inf\limits_{u>0}\sup\limits_{x\in \Pi}\frac{(\textit{L}_\epsilon u)(x)}{u(x)} = \lambda_\epsilon.
\end{align*} 

Our goal is now to apply the reduction we did in subsection \ref{sec:4.1.1} to the new operator $\overset{\sim}{\textit{L}}_\epsilon$ to show that we can characterize, as a viscosity solution, the limit of $W_\epsilon(x):=-\epsilon\ln(v_\epsilon(x))$. The proof is then completed by showing that \( y_0 \) is a significant component and by reapplying Proposition \ref{prop:4.14} to derive the asymptotics of the eigenvalues \( \mu_\varepsilon \). As before, setting $v_\epsilon(x) = e^{-W_\epsilon(x)/\epsilon}$ and substituting it in our modified operator, we get that the {\color{red}principal} eigenvalue $\mu_\epsilon$ verifies
\begin{equation*}
- \epsilon\, a(x) W_\epsilon'' + H( W_\epsilon'(x),x) + \epsilon\, c(x) - k(x) =  \epsilon\, \mu_\epsilon.
\end{equation*}

We pass to the limit via a subsequence, using the fact that \eqref{11} and \eqref{12} respectively become in our new case
\begin{equation*}
    \epsilon\min\limits c(x) - \max\limits k(x)\leq \epsilon\mu_\epsilon \leq \epsilon \max\limits c(x) \;\;\; \text{and}
\end{equation*}
\begin{equation*}
    a(\zeta) (W'_\epsilon)^2(\zeta)
= \epsilon[\mu_\epsilon - c(\zeta)-\frac{k(\zeta)}{\epsilon}] + b(\zeta)W'_\epsilon(\zeta)
\leq [\max\limits c(x) - \min\limits c(x)] + \,C +\, CW'_\epsilon(\zeta).
\end{equation*}

The first equation allows us to extract a subsequence of $\epsilon \mu_\epsilon$ converging to $\Lambda$. The second allows us to get the uniform convergence of $W_\epsilon$ to a limit $W$ on $\Pi$, and to establish that $W$ should be a viscosity solution of the equation
\begin{equation}
\label{23}
    H(W'(x), x) - k(x)= \Lambda \;\;\; \text{on} \;\;\;  \Pi.
\end{equation}

To end the proof, we now apply the viscosity solution theory to get a better characterization and an explicit form of the viscosity solution $W$. It is well known \cite{Fathi2008} (Definition 4.2.6, page 124) that equation \eqref{23} always admits a viscosity solution for a specific constant $\Lambda$ given by 
\begin{align*}
    \Lambda &= -\lim_{T \to \infty} \inf_{\eta \in C^1_p([0,T])} \frac{1}{T} \int_0^T \left[H - k\right]^* (-\eta'(t), \eta(t)) \, dt \\
    &= -\lim_{T \to \infty} \inf_{\eta \in C^1_p([0,T])} \frac{1}{T} \int_0^T  L(-\eta'(t), \eta(t)) + k(\eta(t)) \, dt.
\end{align*}

Thanks to the reduction performed in Subsection  \ref{sec:4.1}, we already know that
\begin{equation*}
    -\lim_{T \to \infty} \inf_{\eta} \frac{1}{T} \int_0^T  L(-\eta'(t), \eta(t)) \, dt = 0.
\end{equation*}

Moreover, using that $k$ is non-negative and null on a neighborhood of $y_0$, we get that $\Lambda = 0$. 

\vspace{\baselineskip}

It follows from \cite{Fathi2008}, Definition \ref{defn:4.7}, and Proposition \ref{prop:4.8} that $W$ verifies
\begin{equation*}
    W(z) = \inf\limits_{x \in A_H } (d_H(z,x) + W(x))
\end{equation*}
    where the new Aubry set is given by 
\begin{equation*}
    A_H := \biggl\lbrace x, \, \text{for  all} \,h > 0 \inf\limits_{T \ge h, \, \gamma \, \in \, S_x^x(T)} \int_0^T L(-\gamma'(s) ,\gamma(s)) \, + k(\gamma(s)) \, ds = 0
\biggl\rbrace,
\end{equation*}
and the new Mañé potential by 
\begin{equation*}
d_H(z,x) = \inf\limits_{t > 0, \, \eta\in S_z^x(t)} \int_0^t L(-\eta'(s),\eta(s)) + k(\eta(s)) \,ds.
\end{equation*}
Finally, suppose that $A_H = \lbrace y_0\rbrace$, then $y_0$ is the only significant component. As $k=0$ on a small neighborhood of $y_0$, we can apply Propositions \ref{prop:4.14} and \ref{prop:4.15}, to deduce that
    \begin{equation*}
        \lim\limits_{\epsilon \rightarrow 0} \mu_\epsilon = -(0\lor b^\prime(y_0))+c(y_0) \;\;\; \text{and  therefore} \;\;\;
    \Lambda(y_0) \leq  \liminf\limits_{\epsilon \rightarrow 0} \lambda_\epsilon.
    \end{equation*} 
    
Moreover, as $y_0$ was an arbitrary zero of $b$, the previous inequality holds for any zeros of $b$. Taking the maximum in the previous inequality yields the desired result. 
\vspace{\baselineskip}

Thus, to complete the proof, it only remains to show that $A_H = \lbrace y_0\rbrace$. First of all, it is clear that $y_0 \in A_H$. Moreover, arguing as in the proof of Proposition \ref{prop:4.8}, it is clear that no point at which the vector field does not vanish can belong to the Aubry set. To conclude, we prove that any zero distinct from \( y_0 \) does not belong to the Aubry set. Let $z_0$ be such a zero. We argue by contradiction, supposing that there exists a sequence of increasing times $(T_n)_n$ and minimizing curves $\gamma_n : [0,T_n] \rightarrow \Pi$ such that $\gamma_n(0) = \gamma_n(T_n) = z_0$ and
\begin{equation*}
    \lim_{n \rightarrow +\infty} \int_0^{T_n} L(-\gamma_n'(s) ,\gamma_n(s)) + k(\gamma_n(s)) \, ds = 0.
\end{equation*}

Moreover, there exists $\epsilon > 0$ sufficiently small such that $0 < \delta_1 \leq {\color{red}k} \leq \delta_2$ holds on $B({\color{red}z_0},\epsilon)$. Let us set $E_n := \lbrace t, \gamma_n(t) \in B(z_0, \epsilon) \rbrace$. Then, we can suppose that, up to a sub-sequence, $|E_n|$ tends to $0$, $T > 0$, or diverges to $+ \infty$. Suppose that we are in the first case. Arguing as in the proof of Proposition \ref{prop:4.8}, we obtain a contradiction, as
\begin{equation*}
    \epsilon \, {\color{red}\leq} \int_{E_n} \gamma_n'(s)-b(\gamma_n(s)) +b(\gamma_n(s)) ds \leq  \left( \int_{E_n} (\gamma_n'(s)-b(\gamma_n(s)) )^2 ds \right)^{\frac12}\sqrt{|E_n|}+\delta_2\,|E_n| \rightarrow 0.
\end{equation*}
Now, suppose that up to a subsequence, $|E_n|$ tends to $T > 0$ or diverges to $+ \infty$. Then, there exist $N$ such that for any $n > N$, $|E_n| > T/2$. A contradiction follows as
\begin{equation*}
    \lim_{n \rightarrow +\infty} \int_0^{T_n} L(-\gamma_n'(s) ,\gamma_n(s)) + k(\gamma_n(s)) \, ds \ge \liminf_{n \to \infty} \int_{E_n} L(-\gamma_n'(s) ,\gamma_n(s)) \, ds + \frac{\delta_1 T}{2} > 0,
\end{equation*}
which completes the proof. 
\end{proof}

\subsubsection{The case of flat zeros}\label{sec:4.1.4}

In this subsection, we complete the proof of Theorem \ref{thm:4.1} in the case where a significant component is a flat zero.

\begin{prop}\label{prop:4.18}
    Let $x_0$ be a significant component that is a flat zero, then 
    \begin{equation*}
        \lim\limits_{\epsilon \rightarrow 0} \lambda_\epsilon = c(x_0).
    \end{equation*}
\end{prop}

\begin{proof}
We assume that $x_0 = 0$ and $W(0) = 0$. We begin with the case when $x_0$ is an attractive zero, which means that $b(x) > 0$ for negative $x$ and $b(x) < 0$ for positive ones. Then the representation formula for viscosity solution gives us, in a small neighborhood of $0$
\begin{equation*}
    W(x) = d_H(x,0) = 0.
\end{equation*}

Considering the test function $W_\delta^{+}(x) = \delta x^2$ with $\delta > 0$, we easily deduce, using that $b(x) = O(x^n)$ holds for any power $n$ in a neighborhood of $0$, that our choice of $W_\delta^{+}$ verifies the assertions in Proposition \ref{prop:4.14}. Thus
\begin{equation*}
   c(0) \leq \lim\limits_{\epsilon \rightarrow 0} \lambda_\epsilon.
\end{equation*}

To show that the upper bound also holds, let us consider two points $\alpha < 0 < \beta$ in the neighborhood of $0$ such that $W(x) = 0$ for any $x \in [\alpha, \beta]$ and such that $b(\beta) < 0 < b(\alpha)$. Let us define test function $W_\delta^{-}$ by
\begin{equation*}
W_\delta^{-}(x) =
\begin{cases}
\mathrel{\phantom{-}}0 & \text{if } x \in [\alpha, \beta], \\
-\delta (x - \alpha)^4 & \text{if } x < \alpha, \\
-\delta (x - \beta)^4 & \text{if } x > \beta.
\end{cases}
\end{equation*}

It is clear that we can always choose $\alpha$ and $\beta$ small enough such that 
\begin{equation*}
W_\delta^{-} = W = 0 \,\,\,  \text{on} \,\,\, [\alpha,\beta] \,\,\, \text{and} \,\,\, W_\delta^-(x) < W(x)  \;\;\; \text{on}  \;\;\; [\alpha -\delta,\alpha) \;\cup\; (\beta,\beta +\delta].
\end{equation*}

Moreover, for all \( x \in [\alpha,\beta] \), one has \( H(W_\delta^{-\prime}(x), x) = 0 \) and for \( \delta \) sufficiently small and for \( x \in [\alpha - \delta, \alpha) \) and \( x \in (\beta, \beta + \delta] \), respectively, we have
\begin{equation*}
    H(W_\delta^{-\prime}(x), x)
= 4\delta(x - \alpha)^3 b(x) + O(\delta^2) \leq 0 \, \text{and} \, H(W_\delta^{-\prime}(x), x)
= 4\delta(x - \beta)^3 b(x) + O(\delta^2) \leq 0.
\end{equation*}

Thus, following the same method as in the proof of Proposition \ref{prop:4.15}, we get
\begin{equation*}
\lim\limits_{\epsilon \rightarrow 0} \lambda_\epsilon \leq\max\limits_{x \in [\alpha, \beta]} c(x).
\end{equation*}

By continuity of the maximum of a continuous function, letting \( \alpha \) and \( \beta \) tend to \( 0 \), we obtain the desired upper bound.

\vspace{\baselineskip}

We now continue with the case where $0$ is neither an attractor nor a repulsor, but attracts trajectories from the left and pushes back from the right. Of course, by symmetry, it is sufficient to prove the result in the case where $b(x)$ is positive for any non-zero $x$. In this case, the representation formula for viscosity solution gives us
\begin{equation*}
    W(x) = d_H(x, 0) =
    \begin{cases}
         {\phantom{-}} 0 & \text{if } x \leq 0 \\
        > 0 & \text{if } x > 0.
    \end{cases}
\end{equation*}

To obtain the desired result, as before, for a small negative \( \alpha \) and a positive \( \delta \), we choose test functions that are locally
\begin{equation*}
W_\delta^{-}(x) =
\begin{cases}
\mathrel{\phantom{-}}0 & \text{if } x \ge \alpha \\
-\delta (x - \alpha)^4 & \text{if } x < \alpha 
\end{cases}
\,\,\, \,\,\, \text{and} \,\,\, \,\,\, W_\delta^{+}(x) = \delta x^2.
\end{equation*}

Thus, using the same method as in the proof of Proposition \ref{prop:4.15}, it follows that 
\begin{equation*}
c(0) \leq \, \lim\limits_{\epsilon \rightarrow 0} \lambda_\epsilon \leq\sup\limits_{x \in [\alpha, 0]} c(x).
\end{equation*}

Passing in the limit as $\alpha$ goes to 0 leads us to the result.

\vspace{\baselineskip}

To end the proof, it suffices to deal with the case where $0$ is a repulsive zero. In this case, the representation formula for viscosity solution gives us
\begin{equation*}
    W(x) = d_H(x, 0) =
    \begin{cases}
         {\phantom{-}} 0 & \text{if } x = 0 \\
        > 0 & \text{if } x \neq 0.
    \end{cases}
\end{equation*}

It suffices to choose as a test function $W_\delta^{-}(x) = 0$ in a small neighborhood of $0$, or to remark that $W_\epsilon$ converges uniformly to $W$, which has a strict minimum at $0$, to deduce the upper-bound. To get the lower one, let us pose $W_\delta^{+}(x) = \delta x^2$ with $\delta$ a small positive parameter. We notice that
\begin{align*}
        H(W_\delta^{+\prime}(x), x) =  4&\delta^2 a(x) x^2 - 2\delta b(x) x = 4\delta^2 a(x) x^2 + O(x^3) \ge 0 \\
        &\text{and} \,\,\, \limsup\limits_{\delta \rightarrow 0} a(0)W_\delta^{+\,\prime\prime}(0)= 0.
\end{align*}

It suffices to show that $ W(x) < W_\delta^+(x)$ locally holds on a punctured ball of 0 to deduce the lower bound. To prove it, we argue as in the proof of Proposition \ref{prop:4.14}, considering a curve $\eta$ that is a solution of $- \eta'(s) + b(\eta(s)) = \alpha \delta \, \eta(s)$ with initial condition $x$. The flatness of $b$ at $0$ implies that the dynamics near $0$ are dominated by the term $\delta \alpha \eta$. This flatness implies that any solution starting close enough to $0$ decays exponentially to $0$, and so does its derivative. Taking $\alpha = 8/a(0)$ and arguing in the same way that in the proof of Proposition \ref{prop:4.14}, we get the desired inequality.
\end{proof}

\subsubsection{Asymptotic of the {\color{red}principal} eigenfunction}\label{sec:4.1.5}

Having established the convergence of the principal eigenvalue, we now turn to the associated eigenfunction giving an analog result as the one in \cite{PiatnitskiRybalko2015} (Theorem 2, page 6). Building on the preceding analysis, we obtain the following convergence result.

\begin{cor}\label{cor:4.19}
Suppose that the maximum in \eqref{9} is attained at only one point $x_0$, then the scaled logarithmic transform of the {\color{red}principal} eigenfunction $W_\epsilon$ converges uniformly up to adding a constant to $W(x) = d_H(x,x_0)$.
\end{cor}

\begin{proof}
    We assume that $x_0 = 0$ and $W(0) = 0$. It follows from our hypothesis that there exists exactly one significant component that is exactly $0$. Moreover, we have already shown that $W_\epsilon$ converges uniformly to $W$ in Proposition \ref{prop:4.5}, and that up to adding a constant $W(x)=d_H(x,x_0)$ in a neighborhood of $x_0$ due to Lemma \ref{lem:4.11}. It follows from Lemma \ref{lem:4.11} and the fact that, in a finite partially ordered set, a unique minimal element is the least element that for every other zero \(y\), we have \(0 \preceq y\). In other words, for all zeros \( y \), one has \( W(y) = d_H(y,0) \). Finally, using in \eqref{14} the triangle inequality, we get for all $x \in \Pi$
  \begin{equation*}
    W(x) = \min\limits_{y, \, b(y)=0 } (d_H(x,y) + W(y))
     = \min\limits_{y,\, b(y)=0 } (d_H(x,y) + d_H(y,0) )
    = d_H(x,0),
\end{equation*}
which finishes the proof.
\end{proof}

\subsection{Explicit calculation of J}

Under hypothesis (H3), using Theorem \ref{thm:4.1}, we can rewrite $J(\mu)$ as
\begin{equation*}
J(\mu) = \sup_{c \,\in \, C^0(\Pi)} \biggl\{ \int_0^1 c \,d\mu - \max\limits_{x, \, b(x)=0 } \left(-(0\lor b^\prime(x)) + c(x) \right) \biggl\}.
\end{equation*}

Let \( \{ x_1, \dots, x_n \} \) denote the zeros of \( b \). 
We first consider measures supported on this set, that is
\[
\mu = \sum_{i=1}^{n} a_i \delta_{x_i} \,\,\, \text{with} \,\,\,  \sum_{i=1}^{n} a_i = 1.
\]

\begin{prop}\label{prop:4.20}
    In the condition described previously
    \begin{equation*}
        J(\sum_{i=1}^{n} a_i \delta_{x_i}) = \sum_{i=1}^{n} a_i (0\lor b^\prime(x_i)) .
    \end{equation*}
\end{prop}

\begin{proof}
Let us pose 
\begin{equation*}
C_i := \lbrace c\in C^0(\Pi), \Lambda(c) = -(0\lor b^\prime(x_i)) + c(x_i) \rbrace \;\, \text{and} \;\, B_i := (0\lor b^\prime(x_i)).
\end{equation*}

Computing we get 
\begin{align*}
    J(\sum_{i=1}^{n} a_i \delta_{x_i}) &= \sup\limits_{c \, \in \,C^0(\Pi) } \biggl\{ \; \sum_{i=1}^{n}  a_i \; c(x_i) - \max\limits_{j \,\in\, \llbracket 1, n \rrbracket } \left(-B_j + c(x_j) \right) \biggl\} \\
    & = \max\limits_{j \,\in\, \llbracket 1, n \rrbracket } \biggl( \sup\limits_{c \,\in\, C_j } \biggl\{\sum_{i=1}^{n}  a_i \; c(x_i) + B_j - c(x_j)\biggl\} \biggl).
\end{align*}

Suppose that the previous maximum is obtained at $j = k$, we get 
\begin{align*}
    J(\sum_{i=1}^{n} a_i \delta_{x_i}) 
     &= \sup\limits_{c \,\in \, C_k} \biggl\{ \sum_{i=1, i\neq k}^{n}   a_i \; c(x_i) + (a_k -1) c(x_k) + B_k\biggl\}\\
    & = \sup\limits_{c \,\in \, C_k} \biggl\{ \sum_{i=1, i\neq k}^{n}   a_i( \; -B_i + c(x_i)) + a_iB_i + (a_k -1) c(x_k) + B_k\biggl\} \\
      & \leq \sup\limits_{c \,\in \, C_k} \biggl\{ \sum_{i=1, i\neq k}^{n}   a_i( \; -B_k + c(x_k)) + a_iB_i + (a_k -1) c(x_k) + B_k\biggl\}\\
    & = \sum_{i=1}^{n} a_iB_i.
\end{align*}

Moreover selecting $c_0 \in \widetilde{C}^0(\Pi) := \bigl\{c\in C^0(\Pi),  c(x_i) = 1 + B_i \text{ for all } i\bigr\}$, we get
\begin{equation*}
    \biggl(\sum_{i=1}^{n}  a_i \; c_0(x_i)\biggl) - \max\limits_{j \in \llbracket 1, n \rrbracket } \left(-B_j + c_0(x_j) \right) = \biggl(\sum_{i=1}^{n}  a_i \; (1 + B_i)\biggl) - \; 1 = \sum_{i=1}^{n} a_iB_i,
\end{equation*}
from which the result follows. 
\end{proof}

To obtain a complete characterization of \(J\), it remains to treat measures whose support is not contained in the zero set of \(b\).

\begin{prop}\label{prop:4.21}
    Let $\mu \in P(\Pi)$ be a probability measure such that  $\operatorname{supp}\mu \nsubseteq \lbrace x_1, \dots, x_n\rbrace $, then
    \begin{equation*}
        J(\mu) = +\infty.
    \end{equation*}
\end{prop}

\begin{proof}

We decompose the measure $\mu$ as follows 
\begin{equation*}
    \mu=\sum_{i=1}^n a_i\,\delta_{x_i}+\alpha\,\mu_0,
\,\,\, a_i\ge 0,\,\,\, \alpha>0,\,\,\, \sum_{i=1}^n a_i+\alpha=1,\
\end{equation*}
where $\mu_0$ designs a {\color{red} non-zero} measure such that $\mu_0(\lbrace x_1, \dots, x_n\rbrace) = 0$. Moreover, we notice that
\begin{align*}
J(\mu) &= \sup\limits_{c \in C^0(\Pi) } \biggl\{ \int c \, d\mu - \Lambda(c)\biggl\} \ge \sup\limits_{c \in \Tilde{C}^0(\Pi) } \biggl\{  \int c \, d\mu - \Lambda(c)\biggl\}\\
&= \sum_{i=1}^{n} a_iB_i + \alpha\sup\limits_{c \in \Tilde{C}^0(\Pi) } \biggl\{ \int c \, d\mu_0 - \Lambda(c)\biggl\} \\
    & = \sum_{i=1}^{n} a_iB_i - \alpha \, + \alpha\sup\limits_{c \in \Tilde{C}^0(\Pi) } \int c \, d\mu_0 = + \infty.
\end{align*}

Indeed, for any {\color{red} non-zero} measure $\mu_0$ such that $\mu_0(\lbrace x_1, \dots, x_n\rbrace) = 0$,
\begin{equation*}
    \sup\limits_{c \in \Tilde{C}^0(\Pi) } \int c \, d\mu_0 = + \infty.
\end{equation*}

To prove this, let us choose a compact set $K$ such that $K\cap\lbrace x_1, \dots, x_n\rbrace = \varnothing $ with \(\mu_0(K) > 0\).
Let us pose a function \(h \in C^0(\Pi)\) such that \(h(x_i) = 1 + B_i\) for all \(i\) and a positive function \(\varphi \in C^0(\Pi)\) with \(\varphi(x_i)=0\) for all \(i\) such that \(\varphi \equiv 1\) on \(K\). For any \(M>0\), we set \(c_M := h + M\varphi\). It is clear that \(c_M \in \widetilde{C}^0(\Pi)\). As
\begin{equation*}
    \int c_M \, d\mu_0 \;\ge\; \int h \, d\mu_0 \;+\; M\,\mu_0(K)
\;\xrightarrow[M\to\infty]{}\; +\infty,
\end{equation*}
the result follows.
\end{proof}

\subsection{Explicit calculation of I}

\begin{prop}\label{prop:4.22}
Let $\mu \in P(\Pi)$ be a probability measure then
    \begin{equation*}
I(\mu) := \lim\limits_{\epsilon \rightarrow 0} I_\epsilon(\mu) = + \infty.
\end{equation*}    
\end{prop}

\begin{proof}
    First of all, let us recall that equation \eqref{8} remains unchanged, thus
    \begin{align*}
    I_\epsilon(\mu) &= \sup\limits_{w \in C^2(\Pi) } - \int_{\Pi} \epsilon \; a(x)( w''(x)+ w'(x)^2) + b(x)w'(x) \, d\mu(x) = \sup\limits_{v \in \Tilde{C}^{1,0}(\Pi)}\phi_{\mu,\epsilon}(v) ,
\end{align*}
where $\Tilde{C}^{1,0}(\Pi) = \lbrace v \in C^1(\Pi) \; with \; \int_{\Pi} v(x) dx= 0  \rbrace$ and $\phi_{\mu,\epsilon}(v) = - \int_{\Pi} \epsilon \, a\left( v'+ v^2\right) + bv\, d\mu.$

\vspace{\baselineskip}
    
We first show the result in the case of a measure that is supported in the set of zeros of $b$. We argue in the same way as in Proposition \ref{prop:2.3} to obtain the desired result as $\mu$ is a measure with at least one isolated atom $x_0=0$. We set $\delta > 0$ small enough such that $\mu_0(B(0,\delta))= \mu_0(0)$ and $\mu_0(B(0,2\delta) \setminus B(0,\delta))= 0$ and define for any $\epsilon > 0$ the periodic function $v_\epsilon$ by
        \[ v_\epsilon = \begin{cases} 
          -\frac{1}{\epsilon } \,sin(x/ \epsilon) & x\in B(0,\delta) \\
          r_\epsilon(x) & x \in B(0,2\delta)\setminus B(0,\delta) \\
          0 & x \in \Pi \setminus B(0,2\delta),
       \end{cases}
    \]
    where $r_\epsilon$ is a well-chosen function that makes the function in $\Tilde{C}^{1,0}(\Pi)$. Then it follows that
\begin{align*}
     I_\epsilon(\delta_{0}) \ge -\epsilon \; a(0)\bigl( v_\epsilon^{\prime}(0) + v_\epsilon(0)^2 \bigl) - \, b(0) \,v_\epsilon(0) =  a(0)/\epsilon \underset{\epsilon\to 0}{\longrightarrow}\; +\infty.
\end{align*}

We now reduce the analysis to measures whose support is not contained in the zero set of $b$. As before, we decompose $\mu$ as follows 
\begin{equation*}
    \mu=\sum_{i=1}^n a_i\,\delta_{x_i}+\alpha\,\mu_0,
\,\,\, a_i\ge 0,\,\,\, \alpha>0,\,\,\, \sum_{i=1}^n a_i+\alpha=1,\
\end{equation*}
where $\mu_0$ is a probability measure such that $\mu_0(\lbrace x_1, \dots, x_n\rbrace) = 0$. As in the second part of Proposition \ref{prop:3.4}, our goal is to make $I_\epsilon(\mu)$ diverge by substituting into \eqref{8} the appropriate test function $v_\epsilon$ given by
\begin{equation*}
  v_\epsilon = \frac{1}{\sqrt[4]{\epsilon}} \, u \;  \;  \; \text{with} \;  \;  \; \int_0^1 u \, dx = 0.
\end{equation*}   

The same calculations as before yield
\begin{align*}
      &\phi_{\mu,\epsilon}(v_\epsilon) = -\, \frac{\alpha}{\sqrt[4]{\epsilon}} \int_0^1 b(x) u(x) \, d\mu_0(x) + O(\sqrt\epsilon).    
\end{align*}

We end the proof by choosing $u$ such that $\int_0^1 b(x) u(x) \, d\mu_0(x)$ is positive and passing to the limit as $\epsilon$ goes to $0$. To show that such a choice of $u$ is possible, it suffices to rule out that \(\int_0^1 b(x)\,u(x)\,d\mu_0(x)=0\) for every zero-mean functions \(u\). Suppose that it holds, then $ b\, d\mu_0$ is proportional to $dx$, which follows that its support is exactly $\Pi$. Thus, assuming that \(b\,d\mu_0 = c\,dx\) for some constant \(c\neq 0\), we get
\begin{equation*}
c \;=\; \displaystyle \int_{\{\,b>0\,\}} b\, d\mu_0 \, \bigl/ \, \displaystyle |\{b>0\}|
\;>\; 0 \,\,\, and \,\,\,
c \;=\; \displaystyle \int_{\{\,b<0\,\}} b\, d\mu_0 \, \bigl/ \, \displaystyle |\{b<0\}|
\;<\; 0,
\end{equation*}
from which the contradiction follows.
\end{proof}

\section{The mixed case}

\subsection{Asymptotic of the {\color{red}principal} eigenfunction}

In Sections \ref{sec:2} and \ref{sec:4}, we have treated this problem, respectively, in the case where $b$ is zero, and in the case where $b$ changes sign on $\Pi$ and has finitely many zeros. In this section, we consider a more general mixed situation: we assume that $b$ changes sign on $\Pi$, has a finite number of isolated zeros, and a finite number of intervals on which it vanishes.

\vspace{\baselineskip}

We aim to prove the following theorem.

\begin{thm}\label{thm:5.1}
   Under the assumptions stated above, the {\color{red}principal} eigenvalue $\lambda_\epsilon$ in \eqref{2} converges as
$\epsilon$ goes to $0$ to
\begin{equation}
\label{24}
\lim\limits_{\epsilon \rightarrow 0} \lambda_\epsilon = \max\limits_{x, \, b(x)=0 } \left(-(0\lor b^\prime(x)) + c(x) \right).
\end{equation}  
\end{thm}

\begin{proof}
    This result is based on the arguments developed throughout the preceding section. We first observe that the results of Proposition \ref{prop:4.5} remain valid in this new context, thus for any sequence \( \epsilon \to 0 \), there exists a subsequence such that \( \lambda_{\epsilon} \) converges to a limit \( \lambda \) and such that the corresponding transformed eigenfunction \( W_{\epsilon} \) converges uniformly on \( \Pi \) to $W$, a viscosity solution of $H(W'(x), x) = 0$ on $\Pi$. We now aim to represent $W$ as in Proposition \ref{prop:4.8}. Since relations \eqref{13}, \eqref{14}, \eqref{15}, and \eqref{16} remain valid, characterizing $W$ amounts to determining the Aubry set in this new context. As before, it is obvious that $\lbrace x, \, b(x) = 0 \rbrace \subset A_H$. To prove the reverse inclusion, we argue as in the proof of Proposition \ref{prop:4.8}. Using the fact that \( b \) changes sign on \( \Pi \), the proof proceeds exactly as before and yields to $A_H \subset \lbrace x, \, b(x) = 0 \rbrace $. We now extend the arguments after Remark \ref{rem:4.10} to this more general case, for which we need the following lemma.

\begin{lem}\label{lem:5.2}\renewcommand{\labelenumi}{\textit{\arabic{enumi}.}}
Under the assumptions stated above, the following two statements hold.
\begin{enumerate}
    \item \( d_H(x,y) = d_H(y,x) = 0 \) if and only if \( x \) and \( y \) 
    belong to the same maximal connected component of the set 
    where \( b \) vanishes.
    \item \( W \) is constant on each interval where \( b \) vanishes.
\end{enumerate}
\end{lem}

\begin{proof}

We have already proved this result for isolated zeros in Remark \ref{rem:4.10}. We begin the proof of $\mathit{1}$, showing that if \( x \) and \( y \) belong to the same interval where $b$ vanishes, then \( d_H(x,y) = d_H(y,x) = 0 \).  Let us set $x, y \in I$, an interval where $b$ is null. We define the sequence \( \gamma_n : [0, T_n] \to \mathbb{R} \) with \( T_n \to \infty \) by the linear constant speed path staying inside the interval $I$
\[
\gamma_n(t) = x + \frac{y - x}{T_n}\, t.
\]

The implication follows as
\[
\int_0^t L(-\gamma'(s),\gamma(s))ds \leq C \int_0^{T_n} |\gamma'_n(t)|^2\, dt
= C \int_0^{T_n} \left|\frac{y - x}{T_n}\right|^2 dt
= C \, \frac{|y - x|^2}{T_n}
\xrightarrow[n \to \infty]{} 0.
\]

Finally, to prove the reversed implication we suppose that \( d_H(x,y) = d_H(y,x) = 0 \) and that $x$ and $y$ belongs to two different maximal connected components. By the same argument as the one following Remark \ref{rem:4.10}, we find an interval \( I \) on which \( b \neq 0 \) and such that \( I \subset A_H \), which yields a contradiction and completes the proof of $\mathit{1}$.

\vspace{\baselineskip}

To prove $\mathit{2}$, suppose that there exists an interval 
\( I \subset \{ x, b(x) = 0 \} \),
and take \( x, y \in I \).
Since \( x \) and \( y \) belong to the same connected component, 
it follows from $\mathit{1}.$ that \( d_H(x,y) = d_H(y,x) = 0 \).
Finally, using the representation formula $W(z) = \min_{x,\, b(x)=0} \big( d_H(z,x) + W(x) \big)$, we obtain \( W(x) \le W(y) \) and \( W(y) \le W(x) \),
from which the equality follows.
\end{proof}

Let us define $\mathcal{I} := \{ x_i \}_{i=1}^n \,\cup\, (I_j)_{j=1}^m$, where \( \{ x_i \}_{i=1}^n \) denotes the finite set of isolated zeros of \( b \), and
\( ( I_j )_{j=1}^m \) denotes the finite family of maximal intervals on which \( b \) vanishes. Since \( W \) is constant on each \( I {\color{red}\subset}  \, \mathcal{I} \), we define
\( W(I) := W(x) \) for any \( x \in I \) and extend the Mañé distance to sets \(I,J\subset \Pi\) by setting
\begin{equation*}
d_H(I,J)
= \inf_{\substack{t>0 \\ \eta \in S_I^J(t)}}
\int_0^t L\bigl(-\eta'(s), \eta(s)\bigr)\,ds,
\end{equation*} 
where $S_I^J(t) = \left\{ \eta \in C^1_p([0,t]) \; ; \; \eta(0) \in I,\ \eta(t) \in J \right\}.$ We define a relation \( \preceq \) on \( \mathcal{I} \) by setting
\[
I \preceq J
\quad \text{if and only if} \quad
W(J) = d_H(J,I) + W(I).
\]

Item \(\mathit{1}.\) of Lemma \ref{lem:5.2} guarantees that the relation \( \preceq \) remains an order relation in this more general setting, preserving transitivity. Moreover, since \( |\mathcal{I}| \) is finite, we can still extract a minimal component, which is either an isolated zero or an interval on which \( b \) vanishes. If the minimal component is of the first type, we get the result arguing as in the previous section. Without loss of generality, we now suppose that the minimal component is the interval \( [\alpha, \beta] \), on which \( W \) vanishes. We first establish the lower bound.    

\vspace{\baselineskip}

Let $a \in (\alpha, \beta)$ and pose $W_\delta^{+}(x) = \delta (x-a)^2$ with $\delta$ a small positive parameter. It is clear that $ W(x) < W_\delta^+(x)$ on a punctured neighborhood of $a$ and with equality at $a$. Moreover, as $H(W_\delta^{+\prime}(x), x) \ge 0$ on this neighborhood, using $W_\delta^{+\prime\prime}(x) = 2 \delta$ we get $ c(a) \leq \lim\limits_{\epsilon \rightarrow 0} \lambda_\epsilon$. As the preceding inequality holds for any $a \in \,\, (\alpha, \beta)$, using the continuity of $c$ we get
\begin{equation*}
\max\limits_{x \, \in \,[\alpha,\beta]} c(x) \leq \lim\limits_{\epsilon \rightarrow 0} \lambda_\epsilon.
\end{equation*}  

We end the proof by explaining how to obtain the upper bound. The proof is similar to the previous ones and relies on the construction of \( W_\delta^{-}(x) \) given in Propositions \ref{prop:4.14} and \ref{prop:4.18}. To obtain the result, we choose \( W_\delta^{-}(x) \) to be equal to zero on the minimal component and strictly smaller than \( W_0 \) in a neighborhood of this component, while satisfying \( H\bigl(W_\delta^{-\,\prime}(x), x\bigr) \le 0 \).

\vspace{\baselineskip}

In the case where the left- or right-hand $n$-th derivative of $b$ is non-zero for $n\ge 2$, the function \( W_\delta^{-}(x) \) is constructed for $x < \alpha$ or $ x > \beta$ through the constructions of Proposition \ref{prop:4.14}. In the case where all left- or right-hand derivatives are zero, the function \( W_\delta^{-}(x) \) is constructed for $x < \alpha$ or $ x > \beta$ through the constructions described in Proposition \ref{prop:4.18}.
Moreover, since $b$ is $C^1(\Pi)$, the constructed functions are of class \( C^2(\Pi) \). Using such test functions, {\color{red} arguing as in Proposition \ref{prop:4.15}} we obtain that there exists {\color{red}$(x_\epsilon)_\epsilon$} a sequence of {\color{red} minimum point of $W_\epsilon - W^-_\delta$} converging to a point $x_0 \in [\alpha, \beta]$ such that  
\begin{equation*}  \lim\limits_{\epsilon \rightarrow 0} \lambda_\epsilon \leq c(x_0).
\end{equation*}  

Roughly bounding it from above by the maximum leads us to the desired upper bound.
\end{proof}
    
\begin{rem}\label{rem:5.3}
    The result of Corollary \ref{cor:4.19} remains valid in this new case. Indeed, if the maximum in \eqref{24} is attained only at a zero of $b$, $I = {x_0}$ or on an interval where $b$ vanishes, $I = [\alpha, \beta]$, $W_\epsilon$ converges uniformly up to adding a constant to $W(x) = d_H(x,I)$. Moreover, if $b$ vanishes everywhere on $\Pi$, then the previous result provides an alternative proof of \eqref{4} and even yields more. Indeed, in this case, $W_\epsilon$ converges uniformly, up to an additive constant, to $W(x)=0$.
\end{rem}

\subsection{Explicit calculation of J}

We now establish an explicit calculation of $J$ in the general mixed case.

\begin{thm}\label{thm:5.4}
    Let us denote by $\lbrace x_i \rbrace_{i\in [1,n]}$ finite number of zeros of $b$ and by $\lbrace I_j \rbrace_{j\in [1,m]}$ the finite number of maximal intervals where $b$ is null. Then, for any probability measure $\mu \in P(\Pi)$ supported in these two sets with the following decomposition
\begin{align*}
    \mu = \sum_{i=1}^{n} a_i \delta_{x_i} + \sum_{j=1}^{m} b_i \mu_{j} \;\;\; with \;\;\; \sum_{i=1}^{n} a_i +\sum_{j=1}^{m} b_j = 1 \;\;\; and\;\;\; supp(\mu_j) \subset I_j
\end{align*}
we have
\begin{equation*}
    J(\mu) = \sum_{i=1}^{n} a_i (0\land b^\prime(x_i)).
\end{equation*}
Moreover, for any probability measure \( \mu \in P(\Pi) \) such that
\( \operatorname{supp} \mu \nsubseteq \{x, \, b(x) = 0 \} \),
\[
J(\mu) = +\infty.
\]
\end{thm}

\begin{proof}
    We begin proving the first equality. Using the convexity of $J$

\begin{equation*}
    J(\mu) \leq \sum_{i=1}^{n} a_i\, J(\delta_{x_i})
    + \sum_{j=1}^{m} b_j\, J(\mu_j) =  \sum_{i=1}^{n} a_i \,(0 \wedge b'(x_i)).
\end{equation*}
    Moreover, taking a test function $c_0$ such that $-(0\land b^\prime(x_i)) + c_0(x_i) = 1$ for any isolated zeros $x_i$ and null on each interval $I_j$, we attain that supremum.
    
\vspace{\baselineskip}

To show the second equality, we pose 
\begin{equation*}
    \mu = \sum_{i=1}^{n} a_i \delta_{x_i} + \sum_{j=1}^{m} b_i \mu_{j} + \alpha \mu_0 \,\,\, where \,\,\, \alpha > 0 \,\,\, and \,\,\, \mu_0 \biggl( \,\bigcup_{i=1}^{n} \lbrace x_i\rbrace \cup \bigcup_{j=1}^{m} I_j\biggl)=0
\end{equation*}
and define $\widetilde{C}^0_1(\Pi) := \bigl\{c\in C^0(\Pi),  -(0\land b^\prime(x_i)) + c_0(x_i) = 1 \text{ for all } i, \max\limits_{x \in I_j} c(x) = 1 \text{ for all } j\bigr\}$. Then
\begin{align*}
    &J(\mu) = \sup\limits_{c \in C_0(\Pi) } \left( \int c \, d\mu - \Lambda(c)\right) \ge \sum_{i=1}^{n} a_iB_i + \alpha\sup\limits_{c \in \widetilde{C}^0_1(\Pi) } \left( \int c \, d\mu_0 - \Lambda(c)\right) \\
    & = \sum_{i=1}^{n} a_iB_i - \alpha \, + \alpha\sup\limits_{c \in \widetilde{C}^0_1(\Pi) } \int c \, d\mu_0 = + \infty,
\end{align*}
as $c$ could be chosen arbitrary large on $supp(\mu_0)$ .
\end{proof}

\subsection{Calculation of I}

We finish this article dealing with the asymptotic behavior of $I_\epsilon(\mu)$ under our mixed conditions. We recall that, in our case, we are interested in calculating the following quantity
\begin{align*}
    I_\epsilon(\mu) = \sup\limits_{v \in \Tilde{C}_0^1(\Pi) } \phi_{\mu,\epsilon}(v) &\,\,\, \text{with} \,\,\, \Tilde{C}^{1,0}(\Pi) = \lbrace v \in C^1(\Pi) \; with \; \int_0^1 v(x) dx= 0  \rbrace \,\,\, \text{and}
    \\ &  \phi_{\mu,\epsilon}(v) = - \int_{\Pi} \left(\epsilon \; a\left( v'+ v^2\right) + bv \right) \, d\mu.
\end{align*}

We assume that the zero set of $b$ contains at least one interval $I \nsubseteq \Pi$. Building on the preceding analysis, we characterize the limit of $I_\varepsilon(\mu)$ and summarize the result as follows. Arguing in the same way as in the second part of the proof of Proposition \ref{prop:4.22}, if $\mu \in P(\Pi)$ is a probability measure such that $\operatorname{supp}(\mu) \not\subset \{x,\, b(x) = 0\}$, then
    \[
        \lim_{\varepsilon \to 0} I_\varepsilon(\mu) = I(\mu) = +\infty.
    \]
    
Hence, we may restrict our study to the case where $\operatorname{supp}(\mu) \subset \{x,\, b(x)=0\}$. If the intersection of $\operatorname{supp}(\mu)$ with the set of isolated zeros of $b$ is nonempty, the first part of the proof of Proposition \ref{prop:4.22} yields
\[
        \lim_{\varepsilon \to 0} I_\varepsilon(\mu) = I(\mu) = +\infty.
\]

Thus, the last case to consider is that of measures whose support is contained in the intervals where $b$ vanishes. In this case, Propositions \ref{prop:2.4} and \ref{prop:2.8} together with Corollaries \ref{cor:2.9} and \ref{cor:2.11} provide us sufficient conditions for divergence and show that equality \eqref{1} is not guaranteed, neither for Dirac masses, nor for certain absolutely continuous or singular measures.

{\color{red}
\section{$\Gamma$-Convergence and Asymptotic Large Deviation Principles}

\begin{defn}
Let $(\mu_n)_{n \in \mathbb{N}}$ be a sequence of measures on the circle $\Pi$. We say that $(\mu_n)_{n \in \mathbb{N}}$ converges \emph{weakly} to $\mu \in P(\Pi)$, and we write $\mu_n \rightharpoonup \mu \quad \text{as } n \to \infty$, if
\[
\lim_{n \to \infty} \int_\Pi f \, d\mu_n = \int_\Pi f \, d\mu
\quad \text{for all } f \in C^0(\Pi).
\]
\end{defn}

\medskip

\begin{defn}
Let $(I_n)_{n \in \mathbb{N}}$ be a sequence of functionals $I_n : P(X) \to (-\infty,+\infty]$, and let $I : P(X) \to (-\infty,+\infty]$. We say that $(I_n)_n$ $\Gamma$-converges to $I$ with respect to the weak topology on $P(\Pi)$, and we note $I_n \xrightarrow{\Gamma} I$, if the following two conditions hold for every $\mu \in P(\Pi)$:

\noindent
(i) (\emph{$\Gamma$-liminf inequality}) For every sequence $(\mu_n)$ such that $\mu_n \rightharpoonup \mu$ weakly in $P(\Pi)$, it holds that
\[
I(\mu) \leq \liminf_{n \to \infty} I_n(\mu_n).
\]

\noindent
(ii) (\emph{$\Gamma$-limsup inequality}) There exists a sequence $(\mu_n)$ such that $\mu_n \rightharpoonup \mu$ weakly in $P(\Pi)$ and
\[
I(\mu) \geq \limsup_{n \to \infty} I_n(\mu_n).
\]

\end{defn}

\begin{prop}\label{111}
Assume that $(I_n)_n$ $\Gamma$-converges to $I$ with respect to the weak topology on $P(\Pi)$. Then, for any open set $O \subset P(\Pi)$,
\[
\limsup_{n\to\infty} \inf_{\mu \in O} I_n(\mu)
\le
\inf_{\mu \in O} I(\mu).
\]
Moreover, for any closed set $F \subset P(\Pi)$,
\[
\liminf_{n\to\infty} \inf_{\mu \in F} I_n(\mu)
\ge
\inf_{\mu \in F} I(\mu).
\]
\end{prop}

\begin{proof}
Let \(O \subset P(\Pi)\) be open. If \(\inf_{\mu \in O} I(\mu)=+\infty\), there is nothing to prove, thus we assume that this quantity is finite. Let us fix \(\eta>0\). By definition of the infimum, there exists \(\mu_\eta \in O\) such that
\[
I(\mu_\eta)\le \inf_{\mu\in O} I(\mu)+\eta.
\]
Since \(I_n \xrightarrow{\Gamma} I\) with respect to the weak topology, there exists a sequence \((\mu_n^\eta)_n\) for which the \emph{$\Gamma$-limsup inequality} is verified. Taking the \(\limsup\) as \(n\to\infty\), we get
\[
\limsup_{n\to\infty} \inf_{\mu\in O} I_n(\mu)
\le \limsup_{n\to\infty} I_n(\mu_n^\eta)
\le I(\mu_\eta)
\le \inf_{\mu\in O} I(\mu)+\eta.
\]
Letting \(\eta>0\) go to zero, we obtain the first desired inequality. We now aim to prove the second one. Let \(F \subset P(\Pi)\) be closed. If $l := \liminf_{n\to\infty}\inf_{\mu\in F} I_n(\mu)=+\infty,$ there is nothing to prove, then assume that this quantity is finite.
Passing to a subsequence if necessary, we may find \(n_k\to\infty\) such that
\[
\inf_{\mu\in F} I_{n_k}(\mu)\to \ell.
\]
For each \(k\), we choose \(\mu_k\in F\) such that
\[
I_{n_k}(\mu_k)\le \inf_{\mu\in F} I_{n_k}(\mu)+\frac1k.
\]
Since \(P(\Pi)\) is compact for the weak topology, up to extraction we may assume that
$\mu_k \rightharpoonup \mu_0 \in F$ in $P(\Pi).$ Finally, by the \(\Gamma\)-liminf inequality,
\[
\inf_{\mu\in F} I(\mu) \leq I(\mu_0)\le \liminf_{k\to\infty} I_{n_k}(\mu_k)\le \liminf_{k\to\infty}\left(\inf_{\nu\in F} I_{n_k}(\nu)+\frac1k\right)=\ell,
\]
and we obtain the result.
\end{proof}

\begin{prop}\label{1111}
    Under assumption \text{(H2)}, the functionals $(I_\varepsilon)$ $\Gamma$-converge to $J$ as $\varepsilon \to 0$, that is
\[
I_\varepsilon \xrightarrow{\Gamma} J \quad \text{as } \varepsilon \to 0.
\]
\end{prop}

\begin{proof}
We first prove (i) (the \emph{$\Gamma$-liminf inequality}). Let $(\mu_\varepsilon)$ be a sequence such that $\mu_\varepsilon \rightharpoonup \mu$ weakly in $P(\Pi)$. Then, for every $c \in C^0(\Pi)$,
\begin{equation*}
    \liminf_{\varepsilon \to 0} I_\varepsilon(\mu_\varepsilon)
    \ge \liminf_{\varepsilon \to 0} \int_{\Pi} c \, d\mu_\varepsilon - \Lambda_\varepsilon(c)
    = \int_{\Pi} c \, d\mu - \lim_{\varepsilon \to 0} \Lambda_\varepsilon(c).
\end{equation*}
Taking the supremum over $c$ in the right-hand side of the inequality, we conclude that
\begin{equation*}
    \liminf_{\varepsilon \to 0} I_\varepsilon(\mu_\varepsilon)
    \ge \sup_{c \in C^0(\Pi)}
    \left\{
        \int_{\Pi} c \, d\mu
        - \Lambda(c)
    \right\}
    = J(\mu).
\end{equation*}

We now prove (ii) (the \emph{$\Gamma$-limsup inequality}). Observe that if $\mu \neq \omega \, ds$, then $J(\mu) = +\infty$. Hence, in this case, the inequality (ii) is trivially satisfied. It remains to consider the case $\mu = \omega \, ds$. In this case, the result follows from the pointwise convergence established previously and from the equality between the functionals $I$ and $J$. Indeed, by choosing the constant sequence $\mu_\varepsilon = \omega \, ds$, we obtain
\[
\limsup_{\varepsilon \to 0} I_\varepsilon(\mu_\varepsilon)
= \limsup_{\varepsilon \to 0} I_\varepsilon(\omega \, ds)
= I(\omega \, ds)
= J(\omega \, ds),
\]
which completes the proof.
\end{proof}

\begin{cor}\label{11113}
Under assumption \text{(H2)}, for every open set $O \subset P(\Pi)$ and closed set $F \subset P(\Pi)$
\[
\liminf_{\varepsilon\to 0}\,\liminf_{T\to\infty} \frac{1}{T}\ln \mathbb{P}(\mu_T^\varepsilon \in O)
\geq
-\inf_{\mu \in O} J(\mu),
\]
\[
\limsup_{\varepsilon\to 0}\,\limsup_{T\to\infty} \frac{1}{T}\ln \mathbb{P}(\mu_T^\varepsilon \in F)
\leq
-\inf_{\mu\in F} J(\mu).
\]
\end{cor}

\begin{proof}

We only prove the first inequality, since the second one follows by an analogous argument. We recall that, for every $\varepsilon>0$, the family $(\mu_T^\varepsilon)_{T\geq 0}$ satisfies a large deviation principle on $P(\Pi)$ with speed $T$ and rate function $I_\varepsilon$. Then, by the large deviation lower bound, for every open set $O \subset X$, one has
\[
\liminf_{T\to\infty} \frac{1}{T}\ln \mathbb{P}(\mu_T^\varepsilon \in O)
\geq
-\inf_{\mu\in O} I_\varepsilon(\mu).
\]

Passing to the limit as $\epsilon$ tends to zero and using Proposition \ref{111} and  \ref{1111} we obtain
\begin{equation*}
    \liminf_{\epsilon\to0} \liminf_{T\to\infty} \frac{1}{T}\ln \mathbb{P}(\mu_T^\varepsilon \in O)
\geq
- \limsup_{\epsilon\to0} \inf_{\mu\in O} I_\varepsilon(\mu) \ge - \inf_{\mu \in O} J(\mu),
\end{equation*}
from which the desired result follows
\end{proof}

\begin{prop}\label{11112}
    Under assumption \text{(H3)}, the functionals $(I_\varepsilon)$ $\Gamma$-converge to $J$ as $\varepsilon \to 0$, that is
\[
I_\varepsilon \xrightarrow{\Gamma} J \quad \text{as } \varepsilon \to 0.
\]
\end{prop}
\begin{proof}
By arguing as in the first part of the proof of Proposition \ref{1111}, we obtain the \emph{$\Gamma$-liminf inequality}. Furthermore, the argument used in the second part of that proof allows us to restrict the analysis to measures whose support is contained in the zero set of $b$. However, since $I$ and $J$ do not coincide on this class of measures, one can no longer rely on a constant sequence to verify the \emph{$\Gamma$-limsup inequality}. Moreover, suppose that, for each zero $x_i$ of $b$, we have constructed a sequence of measures $(\mu_\varepsilon^i)_\varepsilon$ such that $\mu_\varepsilon^i \rightharpoonup \delta_{x_i}$ and 
$\limsup_{\varepsilon\to0} I_\varepsilon(\mu_\varepsilon^i)\leq J(\delta_{x_i})$,
then, for any measure of the form $\mu=\sum_{i=1}^N \alpha_i \delta_{x_i},$ one can show that the sequence of general term $\mu_\varepsilon:=\sum_{i=1}^N \alpha_i \mu_\varepsilon^i$ satisfies the \emph{$\Gamma$-limsup inequality}. Indeed, one has $\mu_\varepsilon \rightharpoonup \mu$, and, by convexity of $I_\varepsilon$,
\[
\limsup_{\varepsilon\to0} I_\varepsilon(\mu_\varepsilon)
\leq
\sum_{i=1}^N \alpha_i \limsup_{\varepsilon\to0} I_\varepsilon(\mu_\varepsilon^i)
\leq
\sum_{i=1}^N \alpha_i J(\delta_{x_i})
= J(\mu).
\]

Thus, we have reduced the problem to the construction of a sequence of measures $(\mu_\varepsilon)_\varepsilon$ satisfying the $\Gamma$-limsup inequality for a Dirac mass at a point where $b$ vanishes. Without loss of generality, we fix an index $i$, assume that $x_i = 0$, and that $0$ is either a zero of order $m$ or a flat zero of $b$. In the hyperbolic case, we pose \(\psi \in C^2(\Pi)\) be such that
\[
\psi(x)=\frac{x^2}{a(0)}\quad \text{for } |x|\le \delta,
\qquad
\psi(x)>0 \quad \text{for } x\neq 0,
\]
and define
\[
\mu_\varepsilon(dx)=\rho_\varepsilon(x)\,dx \,\,\, \text{with}\,\,\,
\rho_\varepsilon(x)=\frac{1}{Z_\varepsilon}
\exp\!\left(-\frac{\theta\,\psi(x)}{2\varepsilon}\right),
\]
where \(\theta>0\) is fixed and \(Z_\varepsilon\) is the normalization constant. Prior to analyzing the asymptotic behavior of \(I_\varepsilon(\mu_\varepsilon)\), we collect asymptotic estimates for the normalization constant \(Z_\varepsilon\) and the first moments of a random variable distributed according to \(\mu_\varepsilon(dx)\), which will be used below. On the region $|x|>\delta$, since $\psi$ is strictly positive, there exists $c>0$ such that
\[
\int_{|x|>\delta} \exp\!\left(-\frac{\theta\psi(x)}{2\varepsilon}\right)\,dx
= O\!\left(e^{-c/\varepsilon}\right),
\]
and where $|x|\le \delta$, using $\psi(x)=x^2/a(0)$, the change of variable $x=\sqrt{\varepsilon}\,y,$ yields
\[
\int_{|x|\le \delta} \exp\!\left(-\frac{\theta x^2}{2a(0)\varepsilon}\right)\,dx
=
\sqrt{\varepsilon}
\int_{|y|\le \delta/\sqrt{\varepsilon}}
\exp\!\left(-\frac{\theta y^2}{2a(0)}\right)\,dy \sim \sqrt{\varepsilon} \int_{\mathbb{R}} \exp\!\left(-\frac{\theta y^2}{2a(0)}\right)\,dy
=
\sqrt{\varepsilon} \, \sqrt{\frac{2\pi a(0)}{\theta}}.
\]

Therefore,
\begin{equation} \label{estimation1}
\int_\Pi \exp\!\left(-\frac{\theta\,\psi(x)}{2\varepsilon}\right)\,dx
\sim
\sqrt{\varepsilon} \, \sqrt{\frac{2\pi a(0)}{\theta}} \,\,\, \text{and}\,\,\, Z_\varepsilon \sim \frac{1}{\sqrt{\epsilon}}\,\sqrt{\frac{\theta}{2\pi a(0)}}.
\end{equation}

Similarly, the moments associated with a random variable distributed according to \(\mu_\varepsilon(dx)=\rho_\varepsilon(x)\,dx\) are asymptotically those of a centered normal distribution with variance \(\frac{a(0)}{\theta}\varepsilon\). In particular, for every \(p \geq 1\), the moments satisfy
\begin{equation} \label{estimation2}
\int_\Pi x^p\,\rho_\varepsilon(x)\,dx \leq \int_\Pi |x|^p\,\rho_\varepsilon(x)\,dx = O\!\left(\varepsilon^{p/2}\right).
\end{equation} 

In order to simplify the analysis of \(I_\varepsilon(\mu_\varepsilon)\), we first establish the following lemma.

\begin{lem}
Let \(\mu = f\,dx\), where \(f\) is a positive \(C^2(\Pi)\) function. Then
\begin{equation*}
I_\varepsilon(\mu) \leq \sup_{v \in {C}^1(\Pi)}
\left\{ - \int_{\Pi} \left(\varepsilon\, a(x)\bigl( v'(x)+ v(x)^2\bigr) + b(x)v(x) \right) \, d\mu(x)
\right\} = \frac{1}{4\varepsilon}\int_\Pi \frac{\bigl(\varepsilon (a f)' - b f\bigr)^2}{a f}\,dx.
\end{equation*}
\end{lem}

\begin{proof}
    The first inequality follows from the reduced formula
\begin{align*}
I_\varepsilon(\mu)
&= \sup_{\substack{u > 0 \\ u \in C^2(\Pi)}}
\left\{
- \int_{\Pi} \frac{\mathcal{L}_\varepsilon u}{u} \, d\mu
\right\}
= \sup_{v \in \widetilde{C}_0^1(\Pi)}
\left\{
- \int_{\Pi} \left(\varepsilon\, a(v'+ v^2) + bv \right) \, d\mu
\right\},
\end{align*}
together with the inclusion \(\widetilde{C}_0^1(\Pi) \subset C^1(\Pi)\). Moreover, taking $\mu = f dx$, we obtain
\begin{equation*}
   \sup_{v \in {C}^1(\Pi)}
\left\{
- \int_{\Pi} \left(\varepsilon\, a(v'+ v^2) + bv \right) \, d\mu
\right\} =  \sup_{v \in {C}^1(\Pi)}
\left\{\int_\Pi \bigl(\varepsilon(af)' - bf\bigr)v\,dx - \varepsilon \int_\Pi af\,v^2\,dx
\right\}.
\end{equation*}

To obtain the second equality, we set
\[
v_0(x) := \frac{\varepsilon (a f)'(x) - b(x) f(x)}{2 \varepsilon a(x) f(x)} \in C^1(\Pi),
\]
 and verify that
\begin{align*}
    \int_\Pi \bigl(\varepsilon(af)' - bf\bigr)v_0\,dx - \varepsilon \int_\Pi af\,v_0^2\,dx =  \frac{1}{4\varepsilon} \int_{\Pi} 
\frac{(\varepsilon (a f)' - b f)^2}{a f}\, dx.
\end{align*}

Moreover, the Young’s inequality leads us to
\begin{align*}
    & (\varepsilon\, (af)' - bf)v =  \frac{(\varepsilon\, (af)' - bf)}{\sqrt{2\epsilon a f}} \sqrt{2\epsilon a f}v \leq \frac{(\varepsilon\, (af)' - bf)^2}{4 \epsilon a f} + \epsilon af v^2.
\end{align*}

Taking the supremum over \(v \in C^1(\Pi)\), we obtain the upper bound. Since equality holds for \(v = v_0\), the result follows.
\end{proof}

Applying the previous lemma with \(\mu = \mu_\varepsilon\), which satisfies its assumptions, we obtain
\begin{align*}
 I_\varepsilon(\mu_\varepsilon) &\le \frac{1}{4\varepsilon}\int_\Pi
\frac{\bigl(\varepsilon (a\rho_\varepsilon)' - b\rho_\varepsilon\bigr)^2}
{a\rho_\varepsilon}\,dx =
\frac{1}{4\varepsilon}\int_\Pi
\frac{1}{a}\bigl(\varepsilon a'
-\frac{\theta}{2}a\psi'
-b\bigr)^2\,\rho_\varepsilon\,dx\\
&=
\frac{1}{4\varepsilon}\int_\Pi
\frac{1}{a}\bigl(\varepsilon a'
+\bigl(-b'(0)-\theta\bigr)x
+O(x^2)\bigr)^2\,\rho_\varepsilon\,dx \\
&=\frac{1}{4\varepsilon}\int_\Pi
\frac{1}{a}\Big(
\varepsilon^2 (a')^2
+2\varepsilon \, a'(-b'(0)-\theta) x
+(-b'(0)-\theta)^2 x^2
+O(\varepsilon x^2)
+O(x^3)
+O(x^4)
\Big)\rho_\varepsilon\,dx.
\end{align*}

Moreover, using \eqref{estimation1} and \eqref{estimation2}, we obtain the following estimates:
\begin{equation*}
    \frac{1}{4\varepsilon}\int_\Pi \frac{(-b'(0)-\theta)^2 x^2}{a}\,\rho_\varepsilon\,dx \sim \frac{(-b'(0)-\theta)^2}{4\varepsilon\,a(0)} \int_\Pi x^2 \rho_\varepsilon(x)\,dx \sim \frac{(-b'(0)-\theta)^2}{4\theta};
\end{equation*}
\begin{equation*}
\frac{1}{4\varepsilon}\int_\Pi \frac{\varepsilon^2 (a')^2}{a}\,\rho_\varepsilon\,dx = O(\varepsilon); \,\,\,\,\,\, 
\frac{1}{4\varepsilon}\int_\Pi \frac{2\varepsilon a'(-b'(0)-\theta)x}{a}\,\rho_\varepsilon\,dx = O(\sqrt{\varepsilon});\,\,\,\,\,\, \frac{1}{4\varepsilon}\int_\Pi \frac{O(\varepsilon x^2)}{a}\,\rho_\varepsilon\,dx = O(\epsilon);
\end{equation*}
\begin{equation*}
    \frac{1}{4\varepsilon}\int_\Pi \frac{O(x^3)}{a}\,\rho_\varepsilon\,dx = O(\sqrt{\epsilon})  \,\,\, \text{and}\,\,\, \frac{1}{4\varepsilon}\int_\Pi \frac{O(x^4)}{a}\,\rho_\varepsilon\,dx = O(\epsilon),
\end{equation*}
that yield
\[
\limsup_{\varepsilon\to0} I_\varepsilon(\mu_\varepsilon)
\le
\frac{(-b'(0)-\theta)^2}{4\theta}.
\]

If \(b'(0) > 0\), taking \(\theta=b'(0)\) gives $\limsup_{\varepsilon\to0} I_\varepsilon(\mu_\varepsilon)\le b'(0).$ Likewise, if \(b'(0) < 0\), taking \(\theta=-b'(0)\) gives $\limsup_{\varepsilon\to0} I_\varepsilon(\mu_\varepsilon)\le 0.$ Hence,
\[
\limsup_{\varepsilon\to0} I_\varepsilon(\mu_\varepsilon)
\le
\max\bigl(0,b'(0)\bigr) = J(\delta_0).
\]

In the degenerate case, it suffices to consider the same sequence of measures. The previous computations remain unchanged and yield
\[
\limsup_{\varepsilon\to0} I_\varepsilon(\mu_\varepsilon)
\le
\frac{\theta}{4}.
\]
Finally, letting $\theta$ tend to $0$ gives the desired result.
\end{proof}

Using the same argument as in the proof of Corollary~\ref{11113}, we deduce the following corollary.
\begin{cor}
Under assumption \text{(H3)}, for every open set $O \subset P(\Pi)$ and closed set $F \subset P(\Pi)$
\[
\liminf_{\varepsilon\to 0}\,\liminf_{T\to\infty} \frac{1}{T}\ln \mathbb{P}(\mu_T^\varepsilon \in O)
\geq
-\inf_{\mu\in O} J(\mu),
\]
\[
\limsup_{\varepsilon\to 0}\,\limsup_{T\to\infty} \frac{1}{T}\ln \mathbb{P}(\mu_T^\varepsilon \in F)
\leq
-\inf_{\mu\in F} J(\mu).
\]
\end{cor}}


\begin{thebibliography}{99}

{\color{red} \bibitem{BertiniGabrielliLandim2023}
L.~Bertini, D.~Gabrielli and C.~Landim,
\textit{Large deviations for diffusions: Donsker and Varadhan meet Freidlin and Wentzell}.
Ensaios Mat. 38, 77--104, 2023.} 

{\color{red} \bibitem{DiGesuMariani2017}
G.~Di Gesù and M.~Mariani,
\textit{Full metastable asymptotic of the Fisher information}.
SIAM J. Math. Anal. 49, 3048--3072, 2017.} 

\bibitem{DonskerVaradhan1975}
M.~D. Donsker and S.~R.~S. Varadhan,
\textit{Asymptotic Evaluation of Certain Markov Process Expectations for Large Time, I}.
Comm. Pure Appl. Math. \textbf{28} (1975), 1--47.

\bibitem{DonskerVaradhan1976}
M.~D. Donsker and S.~R.~S. Varadhan,
\textit{On the Principal Eigenvalue of Second-Order Elliptic Differential Operators}.
Comm. Pure Appl. Math. \textbf{29} (1976), 595--621.

\bibitem{Evans2010}
L.~C. Evans,
\textit{Partial Differential Equations}, 2nd ed.
Graduate Studies in Mathematics, Vol.~19, American Mathematical Society, Providence, 2010.

{\color{red}
\bibitem{FaggionatoGabrielli2010}
A.~Faggionato and D.~Gabrielli,
\textit{A representation formula for large deviations rate functionals of invariant measures on the one-dimensional torus}.
Annales de l'Institut Henri Poincaré, Probabilités et Statistiques \textbf{48} (2012), no.~1, 105--122.}

\bibitem{Fathi2008}
A.~Fathi,
\textit{Weak KAM Theorem in Lagrangian Dynamics}.
Preliminary Version No.~10, ENS Lyon, 2008.

\bibitem{FengKurtz2006}
J.~Feng and T.~G. Kurtz,
\textit{Large Deviations for Stochastic Processes}.
Mathematical Surveys and Monographs, Vol.~131, American Mathematical Society, Providence, 2006.

{\color{red} \bibitem{Brezis2011}
C.~Fonte Sanchez, P.~Gabriel, S.~Mischler,
\textit{On the Krein-Rutman theorem and beyond.}}

\bibitem{Kifer1980}
Y.~Kifer,
\textit{On the principal eigenvalue in a singular perturbation problem with hyperbolic limit points and circles}.
J. Differential Equations \textbf{35} (1980), 108--139.

\bibitem{Kifer1990}
Y.~Kifer,
\textit{Principal eigenvalues, topological pressure, and stochastic stability of equilibrium states}.
Israel J. Math. \textbf{70} (1990), 1--47.

\bibitem{PengZhangZhou2020}
R.~Peng, G.~Zhang, and M.~Zhou,
\textit{Asymptotic behavior of the principal eigenvalue of a linear second order elliptic operator with small/large diffusion coefficient and its application}.
Nonlinear Anal. \textbf{191} (2020), 111607.

\bibitem{PiatnitskiRybalko2015}
A.~Piatnitski, A.~Rybalko, and V.~Rybalko,
\textit{Singularly perturbed spectral problems with Neumann boundary conditions}.
Complex Var. Elliptic Equ. \textbf{61} (2015), 252--274.

\bibitem{ProtterWeinberger1997}
M.~H. Protter and H.~F. Weinberger,
\textit{Maximum Principles in Differential Equations}.
Springer, 1997.

{\color{red}
\bibitem{Raquepas2024}
R.~Raquépas,
\textit{The large-time and vanishing-noise limits for entropy production in nondegenerate diffusions}.
Ann. Inst. Henri Poincaré Probab. Stat. 60, 431--462, 2024.} 
\end{thebibliography}
\end{document}